\documentclass[10pt]{article}
\usepackage[all,pdf]{xy}
\usepackage{amsmath}
\usepackage{amsfonts}
\usepackage{amssymb}
\usepackage{amsthm}
\usepackage{latexsym}
\usepackage{indentfirst}
\usepackage{bm}
\usepackage{bbm}
\usepackage{graphicx}
\usepackage{epsfig}
\usepackage{mathrsfs}
\usepackage{hyperref}
\usepackage{xcolor}

\usepackage{enumerate}

\usepackage{algorithm}
\usepackage{algorithmicx}
\usepackage{algpseudocode}
\usepackage{ulem}
\usepackage{appendix}

\newtheorem{theorem}{Theorem}[section]
\newtheorem{definition}[theorem]{Definition}
\newtheorem{proposition}[theorem]{Proposition}
\newtheorem{lemma}[theorem]{Lemma}
\newtheorem{corollary}[theorem]{Corollary}
\newtheorem{remark}[theorem]{Remark}
\theoremstyle{definition} 
\newtheorem{example}[theorem]{Example}

\usepackage{fancyhdr}

\usepackage{tikz-cd}

\usepackage{tikz}
\newcommand*{\circled}[1]{\lower.7ex\hbox{\tikz\draw (0pt, 0pt)%
		circle (.5em) node {\makebox[1em][c]{\small #1}};}}
\usepackage{geometry}
\newcommand{\ho}{\mathrm{Hom}}

\newcommand{\m}{\mathbbm{m}}
\newcommand{\irr}{\mathrm{Irr}}
\newcommand{\kk}{\mathbbm{k}}
\newcommand{\tphi}{\tilde{\varphi}}
\newcommand{\hh}{\mathrm{HH}}

\begin{document}
	
	\title{\bf\Large The second Hochschild cohomology and deformations of Brauer graph algebras}
	\author{Yuming Liu$^a$, Zhengfang Wang$^b$, and Bohan Xing$^{a,*}$}
	\maketitle
	
	\renewcommand{\thefootnote}{\alph{footnote}}
	\setcounter{footnote}{-1} 
	\renewcommand{\thefootnote}{\alph{footnote}}
	\setcounter{footnote}{-1} 
	\setcounter{footnote}{-1} \footnote{$^a$Yuming Liu and Bohan Xing, School of Mathematical Sciences, Laboratory of Mathematics and Complex Systems, Beijing Normal University,
		Beijing 100875,  P. R. China.}
	\setcounter{footnote}{-1} \footnote{$^b$Zhengfang Wang, School of Mathematics, Nanjing University,
		Nanjing 210093,  P. R. China.}
	\setcounter{footnote}{-1} \footnote{E-mail addresses: ymliu@bnu.edu.cn (Y. Liu); zhengfangw@nju.edu.cn (Z. Wang); bhxing@mail.bnu.edu.cn (B. Xing).}
	\setcounter{footnote}{-1} \footnote{$^*$Corresponding author.}
	
	{\noindent\small{\bf Abstract:} In this paper, we give an explicit description on the second Hochschild cohomology groups of bipartite Brauer graph algebras with trivial grading. Based on this, we provide geometric interpretations of deformations associated to some standard cocycles in terms of the surface models of Brauer graph algebras.}
	
	\section{Introduction}
	
	Brauer graph algebras are a class of finite dimensional algebras which can be defined by quivers with relations. The quiver of a Brauer graph algebra can be constructed from a Brauer graph which is by definition a ribbon graph together with a multiplicity function assigning a positive integer to each vertex. One can embed a ribbon graph into an oriented surface in a minimal way so that there is a natural geometric object associated to a Brauer graph algebra. In this way, Opper and Zvonareva \cite{OZ} have constructed the surface models of Brauer graph algebras and given the derived equivalence classification of Brauer graph algebras.

	The goal of this paper is to give an explicit description of the second Hochschild cohomology of bipartite Brauer graph algebras and to provide geometric interpretations of the associated deformations in terms of the surface models. By {\it bipartite} we mean that the set of vertices in the ribbon graph can be divided into two disjoint subsets such that no edges exist between vertices within the same subset. For some technical reasons, throughout this paper we will work over an algebraically closed field $\Bbbk$ of characteristic zero.

	For a $\Bbbk$-algebra given by a quiver with relations, Chouhy and Solotar \cite{CS} constructed a projective resolution depending on a choice of a reduction system. Using this projective bimodule resolution, a general combinatorial method to compute the Hochschild cohomology and to study the deformation theory of  quiver algebras was developed by Barmeier and Wang in \cite{BW}. We mention that  there seems no systematic way to construct a reduction system for a general Brauer graph algebra. However, for any bipartite Brauer graph algebra, we may construct a natural reduction system. Using this, we can give an explicit description of the second Hochschild cohomology groups of a bipartite Brauer graph algebra. That is, we show that the second Hochschild cohomology group admits a basis consisting of four types of {\it standard cocycles} (denoted by type (A), (B), (C) and (D) respectively), see the explicit forms in Proposition \ref{prop:cocycle}. Furthermore, in Theorem \ref{thm:infinitesimaldeformation}, we demonstrate that each standard cocycle can be realized as a formal deformation of the given bipartite Brauer graph algebra.
	
	In Section \ref{sec:4}, we show that the deformations corresponding to the standard cocycles can be interpreted in terms of the surface models defined in \cite{OZ}. Precisely, for a non-local Brauer graph algebra $B_\Gamma$ associated with the bipartite Brauer graph $(\Gamma,\m)$, we have the following conclusions on the second Hochschild cohomology group $\hh^2(B_\Gamma)$ of $B_\Gamma$.
	\begin{itemize}
		\item There is exactly one standard cocycle of type (A). This is related to the fact that the line field in the surface arises from a vector field for a bipartite Brauer graph algebra (concentrated in degree $0$). We show that the deformation of type (A) is isomorphic to a direct product of matrix algebras (see Proposition \ref{semi A}). In particular, the deformed algebra is semisimple. We also provide examples of (ungraded) non-bipartite Brauer graph algebras (Examples \ref{example:annulus}--\ref{example:annuli-with-2-punctures}) which do not admit a semisimple deformation, where the line field does not arise from a vector field.

		\item The standard cocycles of type (B) depend on the multiplicities of the vertices. Explicitly, for each vertex $v$ with the multiplicity $\m(v)$, there are $\m(v)-1$ standard cocycles of type (B). 
		
		\item The standard cocycles of type (C) correspond to the generators of $\mathrm{H}^1(\Gamma)$, the first cohomology group of the ribbon graph $\Gamma$. Since the ribbon graph $\Gamma$ is a deformation retract of its associated ribbon surface $\Sigma$, these generators coincide with the generators of $\mathrm{H}^1(\Sigma)$, the first cohomology group of $\Sigma$, which can act on the homotopy class of the line fields on $\Sigma$ (see \cite{LP}). Therefore, the deformations associated to this type correspond to the deformations of the line field.

		\item Each bigon in the surface model of the Brauer graph (see for example\ in Figure \ref{fig:(D)-in-geomrtry}) induces two distinct standard cocycles of type (D). Note that the bigons are  in one-to-one correspondence to the boundary components with winding number $-2$. The associated deformations of the two cocycles have the following geometric interpretation respectively 
		\begin{itemize}
			\item Compactifying the boundary component into a smooth point (that is, filling the boundary component by a smooth disk), so that the deformed algebra is still a Brauer graph algebra (up to Morita equivalence) (which is non-basic as the corresponding two simple modules become isomorphic after the deformation) with the compactified surface as its surface model. 
			\item Compactifying the boundary component into a `singular' point. This in some sense resembles an orbifold point discussed in \cite{BSW}, precisely, we expect that for any admissible arc system on this surface, the sequence of morphisms around this boundary induces a deformed higher multiplication on the underlying Brauer graph algebra. 
			In some special case, such a deformation induces a transition from a Brauer graph algebra to a gentle algebra, which can be viewed as a way of removing the trivial extension (Example \ref{exa:BGA-deform-gentle}).
		\end{itemize}
	\end{itemize}
	
The above description also allows us to obtain the dimension of the second Hochschild cohomology group of a bipartite Brauer graph algebra.
	
	\begin{theorem}$(\text{\rm see Corollary \ref{dimhh2} and Example \ref{exam:localbrauer}})$
		Let $B_\Gamma=\kk Q_\Gamma/I_\Gamma$ be a bipartite Brauer graph algebra corresponding to the bipartite Brauer graph $(\Gamma,\m)$.
        \begin{enumerate}
            \item 
       If $B_\Gamma$ is not local, then $$\dim_\kk\hh^2(B_\Gamma)=2+\sum_{v\in V}(\m(v)-1)+|E|-|V|+|\mathrm{S_{2\text{-}cyc}}|,$$
		where $V$ (resp.\ $E$) is the vertex (resp.\ edge) set of $\Gamma$ and 
		$$\mathrm{S_{2\text{-}cyc}}=\{\alpha\beta\;|\;\text{$\alpha\beta$ is a $2$-cycle in $B_\Gamma$ with $\alpha\neq\beta$, $\alpha\beta,\beta\alpha\in I_\Gamma$}\}.$$
        \item If $B_\Gamma$ is local, then $\Gamma$ is a ribbon graph with only one edge and two vertices. Suppose the multiplicities of these two vertices are $m_1$ and $m_2$ respectively.
        \begin{itemize}
            \item If both $m_1\neq 1$ and $m_2\neq 1$, then 
            $$\dim_\kk\hh^2(B_\Gamma)=m_1+m_2+1.$$
            \item Otherwise, $B_\Gamma$ is a local Brauer tree algebra with one exceptional vertex of multiplicity $m$ (that is, the case where $m_1=m$ and $m_2=1$). In this case,
            $$\dim_\kk\hh^2(B_\Gamma)=m.$$
        \end{itemize}
        \end{enumerate}
	\end{theorem}

	For a ribbon graph $\Gamma$, let $\Sigma$ be its associated ribbon surface. The first Betti number of $\Sigma$ is given by $\mathrm{rank}(\mathrm{H}^1(\Sigma)) = |E| - |V| + 1$.
If $B_\Gamma$ is not local then $|\mathrm S_{2\text{-cyc}}| = 2 |\partial\Sigma_{-2}|$ where  $\partial\Sigma_{-2}$ is the set of boundary components in $\Sigma$ with winding number $-2$. (For a local Brauer graph algebra we may have that  $|\mathrm S_{2\text{-cyc}}| > 2 |\partial\Sigma_{-2}|$.)   Therefore, in this case the dimension of $\hh^2(B_\Gamma)$ may be expressed in terms of the surface model:


\begin{align}\label{align:formalhh}
\dim_\kk\hh^2(B_\Gamma)=1+\sum_{v\in V}(\m(v)-1)+ \mathrm{rank}(\mathrm H^1(\Sigma))+ 2|\partial\Sigma_{-2}|.
\end{align}

By Keller \cite{Kel}, deformations of an algebra can be viewed as deformations of its derived category. It is natural to ask whether, for two derived equivalent algebras, deformations induced by the same Hochschild cohomology remain derived equivalent. This question has a positive answer for gentle algebras (see \cite{BSW1}), using geometric models. We expect a similar phenomenon for Brauer graph algebras. However, as their geometric model is less well developed, a complete picture is still out of reach. For now, we obtain analogous results for type (A) deformations, supporting this expectation.
	
	\begin{proposition}$(\text{\rm see Proposition \ref{semi A} and Corollary \ref{pre der}})$
		Let $\Gamma$ and $\Gamma'$ be bipartite Brauer graphs and $B_\Gamma$, $B_{\Gamma'}$ be the corresponding Brauer graph algebras. Denote by $B_\Gamma^{(A)}$ and $B_{\Gamma'}^{(A)}$, the deformed algebras of type (A). If $B_\Gamma$ and $B_{\Gamma'}$ are derived equivalent, then the two algebras $B_\Gamma^{(A)}$ and $B_{\Gamma'}^{(A)}$ are semisimple and Morita equivalent.
	\end{proposition}
   The bipartite case also suggests that many potential phenomena may occur in the deformation theory of Brauer graph algebras.

First, we expect that semisimple deformations as in the above proposition exist whenever the line field in the surface model is orientable, even if the Brauer graph $\Gamma$ is not bipartite; see Example \ref{exa:graded-simple}. We expect that such semisimple deformation exists as long as the line field of the surface model is orientable, even though $\Gamma$ is not bipartite, see Example \ref{exa:graded-simple}. 
	
	Second, although we expect that skew-Brauer graph algebras (see for example in \cite{EGV,Soto}) can be obtained as deformations of Brauer graph algebras---motivated by the fact that skew-gentle algebras arise as deformations of gentle algebras---this phenomenon does not occur for trivially graded bipartite Brauer graph algebras. Nevertheless, Example \ref{example:annuli-with-2-punctures} suggests that it may still appear in some non-bipartite cases.
	
It is worth to note that, in recent studies, Brauer graph algebras are also related with Ginzburg dg algebras. Their relations are given by the following diagram.
	
	$$
	\begin{tikzcd}
		\textbf{Gentle algebras} \arrow[dd, "\text{Trivial extension (\cite[Section 3]{Sch})}"', Rightarrow] \arrow[rrrr, "\text{Koszul dual (\cite[Section 3.3]{BH})}", Leftrightarrow] &  &  && \textbf{Gentle algebras} \arrow[dd, "\text{Calabi-Yau completion (\cite[Section 4]{IQ})}"', Rightarrow] \\
		&  & & &                                \\
		\textbf{Brauer graph algebras} \arrow[rrrr, "\text{``Koszul dual'' (\cite[Section 3]{CHQ})}",  Leftrightarrow]   &  &&  &        \textbf{Ginzburg dg algebras}                      
	\end{tikzcd}$$
	Here, by ``Koszul dual'' we mean that a Brauer graph algebra is Koszul dual to a certain cyclic quotient of the associated Ginzburg dg algebra arising from the surface model, rather than to the Ginzburg dg algebra itself. However, since the Ginzburg dg algebras are differential graded and always infinite dimensional, it is  in general difficult to discuss the deformation theory of them. In some special case of Ginzburg dg algebras, for example, Liu and Wang \cite{LW} have given some description about their $A_\infty$-deformations. Therefore, it is quite natural to consider the deformation theory through Brauer graph algebras. 

	In order to better understand the relationship between deformations of Brauer graph algebras and  Ginzburg dg algebras, it would be natural to study the second Hochschild cohomology of Brauer graph algebras with grading. Such a study, however, would require a thorough understanding of Hochschild cohomology in all degrees for Brauer graph algebras with the trivial grading—a problem that involves substantial and technically demanding computations and is unlikely to be resolved in the near future. Rather than pursuing this full program, the present work takes a more modest but concrete first step: we focus on the  ungraded setting and investigate structural properties of the second Hochschild cohomology.  We hope this contributes toward a longer-term program of exploring deformations of graded  Brauer graph algebras. See Examples \ref{exa:bound of graded} and \ref{exa:graded-simple} for the discussions on the Hochschild cohomology in the graded case.
	
	\medskip
	\textbf{Outline.}\; In Section \ref{sec:2}, we review some basic definitions of bipartite Brauer graph algebras and their surface models. In Section \ref{sec:3}, we construct the reduction systems of bipartite Brauer graph algebras and give the description about their second Hochschild cohomology groups. In particular, for each bipartite Brauer graph algebra $A$, we give a formula of the dimension of $\hh^2(A)$. In Section \ref{sec:4}, we give geometric explanations of these deformations on the surface models of Brauer graph algebras. In Section \ref{sec:5}, we give some examples when Brauer graphs are not bipartite. In Appendix \ref{appendix:2-cocycle} and Appendix \ref{appendix:coboundary}, we compute the $2$-cocycles and $2$-coboundaries from the general forms of $2$-cochains given in Subsection \ref{sec:differential}.

	\section{Brauer graph algebras and surface models}\label{sec:2}
	
	Brauer graph algebras originate in the modular representation theory of finite groups and form an important class of finite dimensional tame algebras (for a survey on Brauer graph algebras, see \cite{SS}). In this section, we give a review on the basic notions of (bipartite) Brauer graph algebras. For details we refer to \cite{OZ,SS}.
	
	Throughout this paper we will concentrate on quiver algebras of the form $\kk Q/I$, where $\Bbbk$ is an algebraically closed field of characteristic $0$, $Q$ is a finite quiver, and $I$ is a two-sided ideal in the path algebra $\kk Q$. For each integer $n\geq 0$, we denote by $Q_n$ the set of all paths of length $n$ and by $Q_{\geq n}$ the set of all paths with length at
	least $n$. We denote by $o(p)$ the origin vertex of a path $p$ and by $t(p)$ its terminus vertex. We will write paths from right to left, for example, $p=\alpha_{n}\alpha_{n-1}\cdots\alpha_{1}$ is a path with the starting arrow $\alpha_{1}$ and the ending arrow $\alpha_{n}$.  
	The length of a path $p$ will be denoted by $l(p)$. 
	
	\subsection{Review on Brauer graph algebras}
	
	\subsubsection{Brauer graphs}
	\begin{definition}
		A ribbon graph is a tuple $\Gamma=(V,H,s,\iota,\rho)$, where
		\begin{enumerate}[(1)]
			\item $V$ is a finite set whose elements are called vertices;
			
			\item $H$ is a finite set whose elements are called half-edges;
			
			\item $s: H\rightarrow V$ is a function;
			
			\item $\iota: H\rightarrow H$ is an involution without fixed points;
			
			\item $\rho: H\rightarrow H$ is a permutation whose cycles correspond to the sets $H_v:=s^{-1}(v)$, $v\in V$.
			
		\end{enumerate}
	\end{definition}
	
	Therefore, every ribbon graph defines a graph with vertex set $V$ whose edges are the orbits of $\iota$. For instance, the edge $\bar h:=\{h, \iota(h)\}$ is incident to the vertices $s(h)$ and $s(\iota(h))$. In fact, every ribbon graph gives rise to a surface by thickening the graph (see for example in \cite[Section 1.1]{OZ}).
	
	For a ribbon graph   $\Gamma=(V,H,s, \iota, \rho)$ we need the following notations. For each vertex $v \in V$ we denote 
	\[
	H_v=\{h \in H \mid  s(h)=v\},
	\]
	the set of half edges incident to $v$ and denote by $\mathrm{val}(v)$ the valency of the vertex $v\in V$, namely
	\[
	\mathrm{val}(v)= |H_v|.
	\]
	In particular, a loop is counted twice in $\mathrm{val}(v).$  Denote by $$\bar h=\{h, \iota(h)\}$$ the edge associated to $h$ in the graph induced by $\Gamma$. The set of all edges is denoted by $E(\Gamma)$.  
	For each half-edge $h\in H$ we denote 
	\[
	h^+ = \rho (h) \;\; \text{and}\;\;  h^-= \rho^{-1}(h)
	\]
	the successor and predecessor of $h$ respectively.
	
	Unless stated otherwise, we will assume that $\Gamma$ is connected, that is, its underlying graph is connected.
	
	\begin{definition}\label{BG}
		A Brauer graph is a pair $(\Gamma,\mathbbm{m})$ consisting of a ribbon graph $\Gamma =(V,H,s, \iota, \rho)$ and a function $\mathbbm{m}: V \rightarrow \mathbb{Z}_{>0}$.
	\end{definition}
	
	That is, a Brauer graph is simply a ribbon graph with  positive integers (that is, multiplicities) assigned to vertices. The function $\mathbbm{m}$ in Definition \ref{BG} is referred to as the multiplicity function and its values as multiplicities. 
	We say that a Brauer graph $(\Gamma,\mathbbm{m})$ is {\it multiplicity-free} if $\mathbbm{m}(v)=1$ for each $v \in V$. In particular,  we call a given vertex $v \in V$ is {\it truncated} if $\m(v)=1 $ and $\mathrm{val}(v)=1$.
	
	\subsubsection{Brauer graph algebras}\label{subsection:def-of-BGA}
	Let $(\Gamma,\mathbbm{m})$ be a Brauer graph and $\Bbbk$ be a base field.  One can associate a quiver $Q=Q_\Gamma$ and an ideal $I=I_\Gamma$ generated by some relations in the path algebra $\mathbbm{k}Q$ as follows. 
	
	\begin{enumerate}[(1)]
		\item The vertices of $Q$ correspond to the edges $E(\Gamma)$ of $\Gamma$. For every half-edge  $h\in H$, there is an arrow in $Q$ $$\alpha_h:\bar{h}\rightarrow\bar{h^+}$$ from the edge $\bar h$ to the edge $\bar{h^+}$ associated to the successor $h^+$ (see Figure \ref{fig1}). That is, $\alpha_h$ may be understood as the angel around the vertex $s(h)\in V$ starting from the half-edge $h$ and ending at $h^+$, which induces a one-to-one correspondence between the set $Q_1$ of arrows and the set $H$ of half edges.  We call the arrow $\alpha_h\in Q_1$  an arrow around the vertex $s(h)$.

		Note that  we have a natural permutation 
		\[
		\sigma \colon Q_1 \to Q_1, \alpha_h\mapsto \sigma(\alpha_h):=\alpha_{h^-}
		\]
		whose orbits are in bijection with $V$. Hence to each arrow $\alpha= \alpha_h$ we may associate a cycle
		$$C_\alpha=\alpha\sigma(\alpha)\cdots\sigma^l(\alpha)$$
		where $l+1$ $(=\mathrm{val}(s(h)))$ denotes the cardinality of the $\sigma$-orbit of $\alpha$ (see Figure \ref{fig1}). Every vertex (e.g. corresponding to the edge $\bar h$) of $Q$ is the starting point of exactly two cycles (e.g. $C_{\alpha_h}$ and $C_{\alpha_{\iota(h)}}$). For simplicity, we define the multiplicity of the cycle $C_{\alpha_h}$ as the multiplicity of the vertex $s(h)$ and write $$\m(C_{\alpha_h}) =\m(s(h)),$$ with a slight abuse of notation. (Note that we adopt the sign rule in \cite{OZ}, that is, $\alpha_h$ is starting from $\bar{h}$, and the ending arrow of $C_{\alpha_h}$ is $\alpha_h$.)

		\begin{figure}[ht]
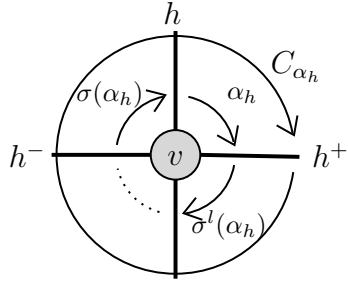

			\centering
			
			\tikzset{every picture/.style={line width=0.75pt}} 
			


			\caption{Arrows in the quiver $Q_\Gamma$.}
			\label{fig1}
		\end{figure}
		
		\item The ideal $I_\Gamma$ is generated by the following set of relations:
		\begin{itemize}
			\item For any composable arrows $\alpha$ and $\beta$ such that  $ \beta \neq \sigma(\alpha)$ (see the left of Figure \ref{fig2}), we have the monomial relation $$\alpha\beta=0.$$
			This type of relation is similar to the one in gentle algebras. 
			
			\item For each edge $\bar h$, we have the relation $$C_\alpha^{\m(C_\alpha)}=C_\beta^{\m(C_\beta)},$$
			where $\alpha,\beta\in Q_1$ and $t(\alpha)=t(\beta) =\bar h$, that is $\alpha$ and $\beta$ end at the same edge $\bar h \in E(\Gamma)$  (see the right of Figure \ref{fig2}). Here $$C_\alpha^{\m(C_\alpha)}= \underbrace{C_\alpha C_\alpha\dotsb C_\alpha}_{\m(C_\alpha)}$$ and $\m(C_\alpha)$ is the multiplicity of the vertex which the cycle $C_\alpha$ is around. 
			
			\begin{figure}[ht]
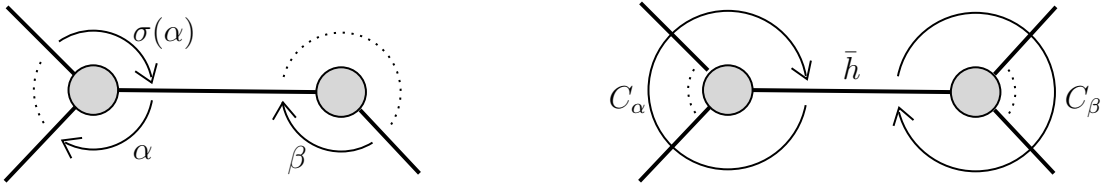

				\centering

				\tikzset{every picture/.style={line width=0.75pt}} 
				

				\caption{The two types of relations in $I_\Gamma$.}
				\label{fig2}
			\end{figure}
			
		\end{itemize}
	\end{enumerate}
	
	The resulting finite dimensional algebra $\mathbbm{k}Q_\Gamma/I_\Gamma$ will be denoted by $B_\Gamma$. Indeed, $B_\Gamma$ is special biserial (see for example\ \cite{SS}), in particular, at each vertex in the quiver $Q_\Gamma$ there are at most two incoming arrows and at most two outgoing arrows. 

	Let $\gamma=\alpha\sigma(\alpha)\cdots\sigma^i(\alpha)$ be a non-trivial path in $B_\Gamma$. We will denote by $\gamma^*$ the path of $B_\Gamma$ which complements $\gamma$ to a maximal non-zero path, that is $\gamma\gamma^*=C_\alpha^{\m(C_\alpha)}$.
	
	\begin{definition}
		A $\mathbbm{k}$-algebra $B$ is called a Brauer graph algebra if there exists a Brauer graph $(\Gamma,\m)$ such that $B\cong B_\Gamma$ as $\mathbbm{k}$-algebras.
	\end{definition}
	
	\begin{remark}\label{remark:truncated}
		\begin{enumerate}[(1)]
		\item Note that for each truncated vertex $v$ (that is,\ $\m(v)=1 =  \mathrm{val}(v)$) in $\Gamma$, there is only one arrow $\beta$ around $v$ so that  the second type of relation becomes non-admissible: $C_\alpha^{\m(C_\alpha)}=\beta$. We usually drop this relation by deleting the loop $\beta$ from the quiver $Q_\Gamma$  and replacing $\beta$ by $C_\alpha^{\m(C_\alpha)}$ in the first type of relation, for instance $\beta \gamma=0$ is replaced by  $C_\alpha^{\m(C_\alpha)} \gamma=0$. 
	\item According to the preceding remark, a Brauer graph algebra $A = kQ_{\mathrm{adm}}/I_{\mathrm{adm}}$ can also be defined in terms of its Gabriel quiver $Q_{\mathrm{adm}}$ and an admissible ideal $I_{\mathrm{adm}}$. The quiver $Q_{\mathrm{adm}}$ is obtained from $Q_\Gamma$ by deleting all loops corresponding to truncated vertices. The admissible ideal $I_{\mathrm{adm}}$ is generated by the following three types of relations:

\begin{enumerate}
  \item[\textnormal{(I)}] $\alpha\beta = 0$ for all paths $\alpha\beta$ with $\beta \neq \sigma(\alpha)$;
  
  \item[\textnormal{(II)}] $C_\alpha^{m(C_\alpha)} = C_\beta^{m(C_\beta)}$,
  whenever $t(\alpha) = t(\beta) = \bar h$ and the two vertices incident to the edge $\bar h$ are both non-truncated;
  
  \item[\textnormal{(III)}] $C_\alpha^{m(C_\alpha)} \alpha = 0$, for any arrow $\alpha$ in $Q_{\mathrm{adm}}$.
\end{enumerate}
(For explicit definitions of the notation, see for example \cite{SS}.)
We refer to these three classes of relations as the {\it defining relations of type I, II, and III}, respectively. According to \cite[Proposition 4.3]{LX}, these relations form a Gr\"{o}bner basis (not necessarily reduced) of the admissible ideal $I_{\mathrm{adm}}$.
       \end{enumerate}
\end{remark}

	We recall that a ribbon graph $\Gamma$ is {\it bipartite} if its set of vertices $V$ admits a partition $$V=V_1\sqcup V_2$$ into disjoint subsets $V_1$ and $V_2$ such that every edge connects a vertex in $V_1$ with a vertex in $V_2$. In other words, no edges connect verteices in $V_1$ (resp.\ in $V_2$).  By \cite[Theorem 4.7]{BM}, a graph is bipartite if and only if every cycle of $\Gamma$ has even length. For convenience, denote the bipartite graph by $\Gamma[V_1,V_2]$. We call a Brauer graph algebra is bipartite if its associated Brauer graph is bipartite.
	Actually, by \cite[Theorem 7.12]{OZ}, bipartite Brauer graph algebras are closed under derived equivalence. 
	
	\subsection{Brauer graph $A_\infty$-categories and surface models}
	
	In this section, we recall some definitions of the Brauer graph $A_\infty$-categories associated with Brauer graphs in graded surfaces introduced in \cite[Section 4]{OZ}. 
We refer the reader to \cite{HKK,KS,LP,OZ} for more details on  the $A_\infty$-categories associated to surface models. We begin with the definition of modified Brauer graph algebras.
	
	\begin{definition}\label{defintion:modified}
		Let $(\Gamma,\m)$ be a Brauer graph and let $\omega:E(\Gamma)\rightarrow\mathbb{Z}$ be a function which assigns an integer to each edge of $\Gamma$. The modified Brauer graph algebra, associated to $(\Gamma,\m,\omega)$ is defined as $B\cong \kk Q_\Gamma/I_\Gamma^\omega$, where $I_\Gamma^\omega$ is generated by the following relations:
		\begin{itemize}
			\item $\alpha\beta$,
			where $\alpha,\beta\in Q_1$ are composable and $\sigma(\alpha)\neq \beta$;
			
			\item $C_\alpha^{\m(C_\alpha)}-(-1)^{\omega(\bar{h})}C_\beta^{\m(C_\beta)}$,
			where $\alpha,\beta\in Q_1$, $t(\alpha)=t(\beta)$, $\alpha=\alpha_{h^-}$ and $\beta=\alpha_{\iota(h)^-}$.
		\end{itemize}
		If $\omega=0$, we recover the usual definition of a Brauer graph algebra as in Section \ref{subsection:def-of-BGA}.
	\end{definition}
	
	We recall from \cite{OZ} the notions on punctured surfaces and arc systems below.
	\begin{definition}
		A punctured surface is a pair $(\Sigma, \mathscr{P})$ where $\Sigma$ is a connected and compact oriented surface with non-empty boundary $\partial\Sigma$ and $\mathscr P$ is a finite set of  interior points (called punctures) in $\Sigma.$
	\end{definition}
	
	\begin{definition}
		An arc system on $(\Sigma, \mathscr{P})$  is a finite set $\mathcal{A}=\{\gamma_1,\cdots,\gamma_n\}$ of simple arcs in $\Sigma$ starting and ending at $\mathscr P$ such that for all $i\neq j$, $\gamma_i$ and $\gamma_j$ are non-homotopic and only intersect transversely in $\mathscr{P}$.
	\end{definition}
	
	For any arc system $\mathcal{A}$ of $\Sigma$, we denote by $|\mathcal{A}|\subseteq\Sigma$ the union of all its arcs. Note that to $\mathcal A$ we may associate a ribbon graph $\Gamma_{\mathcal A}$ whose vertices are $\mathscr P$ and whose edges are $\mathcal A.$
	
	\begin{definition}
		An arc system $\mathcal{A}$ is called full if $\Sigma \setminus |\mathcal{A}|$ is a disjoint union of discs with open boundaries and annuli, where each annulus contains exactly one boundary component $B$ of $\Sigma$ such that $B \cap \mathscr{P} = \emptyset$.

		Furthermore, the arc system is said to be admissible if for every puncture $p \in \mathscr{P}$, there exists an embedded path
		$$c_p:[0,1]\rightarrow (\Sigma\backslash|\mathcal{A}|)\cup\{p\},$$
		such that $c_p(0)=p$ and $c_p(1)\in \partial\Sigma$.
	\end{definition}

	\begin{definition}
		Let $(\Sigma,\mathscr{P})$ be a punctured surface. A line field is a smooth section $\eta:\Sigma\backslash\mathscr{P}\rightarrow \mathbb{P}(T\Sigma)$, where $\mathbb{P}(T\Sigma)$ denotes the projectivized tangent bundle of $\Sigma$. The triple $(\Sigma,\mathscr{P},\eta)$ is called a graded (punctured) surface. Two line fields $\eta_0$ and $\eta_1$ are called homotopic ($\eta_0\simeq\eta_1$) if they are homotopic as maps $\Sigma\backslash\mathscr{P}\rightarrow\mathbb{P}(T\Sigma)$.
	\end{definition}

	We illustrate the above concepts with the following example.

	\begin{example}
		Consider the disk $\Sigma$ with one boundary component as shown in Figure \ref{fig:simple-example-of-arc-sys}, where the set of punctures $\mathscr{P}$ consists of two distinct interior points within $\Sigma$. 
				\begin{figure}[ht]
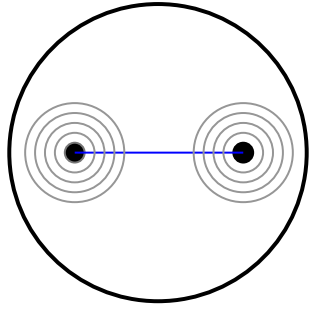

			\centering

\tikzset{every picture/.style={line width=0.75pt}} 



			\caption{An arc system on a disk with two punctures}
			\label{fig:simple-example-of-arc-sys}
		\end{figure}
		An arc system $\mathcal{A}$ on this punctured surface is given by the blue arc. If $\mathcal{A}$ is viewed as a ribbon graph $\Gamma_\mathcal{A}$, then $\Gamma_\mathcal{A}$ can be regarded as a deformation retract of $\Sigma$. This arc system is full because $\Sigma\backslash|\mathcal{A}|$ is an annulus with one open boundary and one closed boundary. A line field on it can be interpreted as the equipotential lines (as illustrated by the gray loops) generated by placing equal and opposite charges at the two punctures.
	\end{example}
	
Actually, each ribbon graph $\Gamma$ corresponding to an admissible arc system in $\Sigma$ naturally induces a line field $\eta_\Gamma$, called a line field of ribbon type (\cite[Example~3.7]{OZ}). In \cite[Section~3.2.3]{OZ}, given a line field $\eta$, arcs can be endowed with an additional grading structure, which allows us to assign to each intersection point
\[
p \in \gamma \vec{\cap} \delta,
\]
(where $p \in \mathscr{P}$ and, under the orientation of $\Sigma$, $p$ is an intersection from the arc $\gamma$ to the arc $\delta$) an integer called the degree of $p$. Such an arc is referred to as a graded arc.

Moreover, in \cite[Section~3.3]{OZ}, an integer called the winding number of an arc $\gamma$, denoted $\omega(\gamma):=\omega_\eta(\gamma)$, is also introduced. This winding number determines the function $\omega$ appearing in the definition of modified Brauer graph algebras.

After all these settings, we can finally define the $A_\infty$-categories associated with (modified) Brauer graph algebras in \cite{OZ}.
	
	\begin{definition}\label{BGC}
		Let $\mathcal{A}$ be a graded admissible arc system on a graded punctured surface $(\Sigma,\mathscr{P},\eta)$ with a line field of ribbon type. Let $\m$ be a multiplicity function on $\mathscr{P}$. The Brauer graph category $\mathbb{B}=\mathbb{B}(\mathcal{A},\m)$ is the following $A_\infty$-category.
		\begin{itemize}
			\item \textbf{Objects:} The set $\{X_\gamma\}_{\gamma\in\mathcal{A}}$ is the set of objects of $\mathbb{B}$.
			
			\item \textbf{Morphisms:} Denote by $B_\mathcal{A}$ the modified Brauer graph algebra of $(\Gamma_\mathcal{A},\m,\omega)$, where $\omega(\gamma):=\omega_\eta(\gamma)$ for all $\gamma\in\mathcal{A}=E(\Gamma_\mathcal{A})$. An arc $\gamma$ corresponds to a vertex of the quiver $Q_{\Gamma_\mathcal{A}}$. Given graded arcs $\gamma,\delta\in\mathcal{A}$, the set of equivalence classes of paths (that is, paths under the canonical surjection 
$\pi: \Bbbk Q_{\Gamma_\mathcal{A}} \twoheadrightarrow B_\mathcal{A}$) from the vertex 
$\gamma$ to the vertex $\delta$ in $B_\mathcal{A}$ forms a $\Bbbk$-basis of 
$\operatorname{Hom}_{\mathbb{B}}(X_\gamma, X_\delta)$.

			\item \textbf{Gradings:} Every path is a homogeneous morphism whose degree equals to the sum of the degrees of its arrows. An arrow $a\in\ho_\mathbb{B}(X_\gamma,X_\delta)$ which corresponds to an intersection $p\in\gamma\vec{\cap}\delta$ at a puncture has degree $|a|:=deg(p)$. The degree of the loop corresponding to a pair of half edges $(h,h)$ at a vertex with valency $1$ is $0$.
			
			\item \textbf{Composition:} For classes of paths $a\in\ho_\mathbb{B}(X,Y)$, $b\in\ho_\mathbb{B}(Y,Z)$ set
			$$\mu_2(b,a):=(-1)^{|a|}ba,$$
			where $ba$ denotes the associated product in $B_\mathcal{A}$.
			
			\item \textbf{Higher operations:} Let $D$ be a disc with a set $\mathcal{M}\subseteq\partial D$ of $n$ marked points. By a marked disc on $\Sigma$ we mean a continuous map $D\rightarrow\Sigma$ which is a smooth immersion on $D\backslash\mathcal{M}$ and which sends points in $\mathcal{M}$ to points in $\mathscr{P}$ and boundary arcs of $D$ to arcs of $\mathcal{A}$. Denote by $p_i\in\delta_i\vec{\cap}\delta_{i+1}$ $(i\in\mathbb{Z}_n)$ its associated sequence of oriented intersections between arcs of $\mathcal{A}$. The corresponding sequence of morphisms $a_n,\cdots,a_1$ will be called a disc sequence. Let $b$ be a path, then
			$$\mu_n(ba_n,\cdots,a_1)=b, \;\text{for $ba_n\neq0$};$$
			$$\mu_n(a_n,\cdots,a_1b)=(-1)^{|b|}b, \;\text{for $a_1b\neq0$};$$
			$$\mu_n(a_n,\cdots,a_{r+1},a_r(ba_r)^*,ba_r,a_{r-1},\cdots, a_2)=(-1)^\circ a_1^*,\;\text{for $ba_r\neq 0$,}$$
			where $\circ=|a_1|+|a_2|\cdots+|a_{r-1}|+|ba_r|+\omega_\eta(\delta_2)+\cdots+\omega_\eta(\delta_r)$.
			Higher operations vanish for all sequences of elements which are not of the form above.
		\end{itemize}
	\end{definition}

	If we regard punctures as vertices and arcs as edges, then each admissible arc system $\mathcal{A}$ of a punctured surface can be regarded as a ribbon graph $\Gamma_\mathcal{A}$. The corresponding Brauer graph $A_\infty$-category $\mathbb{B}(\mathcal{A}, m)$ can be regarded as a generalization of the Brauer graph algebra $B_{\Gamma_\mathcal{A}}$.

	\section{The second Hochschild cohomology groups of bipartite Brauer graph algebras}\label{sec:3}
	
	\subsection{Reduction systems of Brauer graph algebras}\label{subsection:reductionsystem}
	
	We first recall from \cite[Section 1]{Ber} and \cite[Section 2]{CS} the definition of a reduction system for a quiver with relations.
	
	\begin{definition}\label{def:red-sys} Let $Q$ be any finite quiver. 
		A reduction system $R$ for $\kk Q$ is a set of pairs
		$$R=\{(s,\varphi_s)\;|\; \text{for each $s\in S$ pick one element $\varphi_s\in \kk Q$}\}$$
		where
		\begin{itemize}
			\item $S$ is a subset (called tips) of $Q_{\geq 2}$ such that $s$ is not a subpath of $s'$ when $s\neq s'\in S$;
			
			\item for all $s\in S$, $s$ and $\varphi_s$ are parallel, that is, $o(s)=o(\varphi_s)$ and $t(s)=t(\varphi_s)$;
			
			\item for each $(s,\varphi_s)\in R$, $\varphi_s$ is irreducible, that is, it is a linear combination of irreducible paths.
		\end{itemize}
		Here a path is irreducible if it does not contain elements in $S$ as a subpath and we denote by $\irr_S=\irr_S(Q)$ the set of all irreducible paths.
		
		Given a two-sided ideal $I$ of $\kk Q$, we say that a reduction system $R$ satisfies the diamond condition ($\diamond$) for $I$ if the following two conditions hold:
		
		\begin{enumerate}[(i)]
			\item $I$ is equal to the two-sided ideal generated by the set $\{s-\varphi_s\}_{(s,\varphi_s)\in R}$;
			
			\item every path $p$ is reduction-unique, meaning that for any distinct reduction sequences of $p$ via $R$, the resulting irreducible paths are all equal. (For details, one may refer to \cite[Definition 3.3]{BW}.)
		\end{enumerate}
		
	\end{definition}
	
	By Bergman's Diamond Lemma (see \cite[Theorem 1.2]{Ber}), when a reduction system $R$ satisfies the diamond condition ($\diamond$) for $I$, the set $\irr_S=\irr_S(Q)$ of all irreducible paths gives a $\kk$-basis of $\kk Q/I$ through the natural projection $\kk Q \twoheadrightarrow \kk Q/I$.

The reduction system is closely related to the notion of Gr\"{o}bner basis introduced in \cite{Green}. Indeed, a reduced Gr\"{o}bner basis induces a reduction system which satisfies the condition ($\diamond$). However, in general a Gr\"{o}bner basis does not induce immediately a reduction system (in the above sense) since some tips may be contained in other tips as a subpath. Note also that for any Brauer graph algebra $B_\Gamma=\Bbbk Q_\Gamma /I_\Gamma$, a Gr\"{o}bner basis was constructed in \cite[Proposition 4.3]{LX}, which includes the two types of generating relations (as shown in Subsection \ref{subsection:def-of-BGA}), along with some redundant relations. This Gr\"{o}bner basis does not induce a reduction system since some tips are contained in other tips.
	
	In the following, we show that for each bipartite Brauer graph algebra there is a natural way to construct a reduction system which makes the computation for the Hochschild cohomology and deformations explicit. 
	
	Now we define a reduction system for a given bipartite Brauer graph algebra. Let $\Gamma = \Gamma[V_1, V_2]$ be a bipartite Brauer graph, and let $B_\Gamma = \Bbbk Q_\Gamma / I_\Gamma$ be the associated Brauer graph algebra. Note that, in general, the ideal $I_\Gamma$ is not necessarily admissible (see Remark~\ref{remark:truncated}). For our purpose, we give another presentation of $\kk Q_\Gamma/I_\Gamma$ as follows. We remove the loops in $Q_\Gamma$ that are attached to truncated vertices in $V_1$, together with the corresponding relations, as explained in Remark~\ref{remark:truncated}, while keeping the loops attached to vertices in $V_2$. We denote the resulting quiver and ideal by $Q_\Gamma'$ and $I_\Gamma' := I_\Gamma \cap \Bbbk Q_\Gamma'$, respectively. Then from Remark \ref{remark:truncated} we have the following result. 
	
	\begin{lemma}\label{iso}
		We have an isomorphism	of  $\Bbbk$-algebras $$\kk Q_\Gamma/I_\Gamma\cong \kk Q'_\Gamma/I'_\Gamma.$$ 
	\end{lemma}
	
	
	
	Let us construct a reduction system for $I_\Gamma'$. For this, consider  $R_\Gamma$ which consists of following elements $(s, \varphi_s)$ of three types:
	
	\begin{enumerate}[(a)]
		\item $(C_{\alpha}^{\m(C_{\alpha})}, C_{\beta}^{\m(C_{\beta})})$, for each edge $\bar h$ connecting a vertex $v\in V_1$ to  $w\in V_2$ where $C_{\alpha}$ (resp.\ $C_\beta$) is the cycle starting from $\bar h$ around $v$ (resp.\ $w$);
		
		\item $(C_{\alpha_h}^{\m(s(h))}\alpha_h, 0)$, for each half-edge $h$ such that $s(h)\in V_2$;
		
		\item $(\alpha\beta,0)$, where $\alpha,\beta$ are composable arrows in $Q_\Gamma'$ and $\sigma(\alpha)\neq \beta$.
	\end{enumerate}

	\begin{remark}
		Note that the element $(s, \varphi_s)$ of type (a) or (c), which corresponds to the relation $s - \varphi_s = 0$ in $B_\Gamma$, is precisely induced by the two types of generating relations in $I_\Gamma'$, following the definition of a Brauer graph algebra in Subsection \ref{subsection:def-of-BGA}.
		
		For any Brauer graph algebra (not necessarily bipartite) we have the redundant relation $$C_{\alpha_h}^{\m(s(h))}\alpha_h=0$$ for each half-edge $h$.  To obtain a reduction system satisfying the diamond condition, it is necessary to include some of these redundant relations into $R_\Gamma$. However, for general Brauer graph algebras, there appears to be no systematic method for selecting these additional relations to ensure a reduction system satisfying the diamond condition.
	\end{remark}
	\begin{proposition}\label{prop:reductionsystem}
		The reduction system $R_\Gamma$  for $\kk Q'_\Gamma$ satisfies the diamond condition ($\diamond$) for $I'_\Gamma$.
	\end{proposition}
	\begin{proof}
		It is clear that all the $s$ are contained in $Q_{\geq 2}$ and $I_\Gamma'$ is generated by  the set $\{s-\varphi_s\}_{(s,\varphi_s)\in R_\Gamma}$. We first show that the set $\{s-\varphi_s\}_{(s,\varphi_s)\in R_\Gamma}$ is a reduced Gr\"obner basis of $I_\Gamma'$. Using the same method as in \cite[Example 2.2]{Green}, we can give a weight-lexicographic order to make each cycle around the vertex in $V_1$ bigger than cycles around the vertices in $V_2$. For example, consider the commutative relation $C_\alpha^m - C_\beta^n$ in a bipartite Brauer graph algebra $B_\Gamma$. Assume that $C_\alpha$ is a cycle around a vertex in $V_1$ and $C_\beta$ is a cycle around a vertex in $V_2$. By using a weight-lexicographic order, for instance by assigning weights to the arrows in $C_\alpha$ that are significantly larger than those assigned to the arrows in $C_\beta$, we can ensure that the cycle $C_\alpha^m$ around $V_1$ always appears as the tip in the commutative relation $C_\alpha^m - C_\beta^n$. The choice of such an order results in $C_\beta^n$ being an element of the $k$-basis of $B_\Gamma$, and leads to the elements of the form $C_\beta^n \beta(=\alpha^*\cdot\alpha\beta-(C_\alpha^m - C_\beta^n)\cdot\beta)$ lying in this expected Gr\"obner basis.

Then $\{s-\varphi_s\}_{(s,\varphi_s)\in R_\Gamma}$ becomes tip-reduced. Therefore, same as the verification in \cite[Proposition 4.3]{LX}, $\{s-\varphi_s\}_{(s,\varphi_s)\in R_\Gamma}$ is actually a reduced Gr\"{o}bner basis of $I'_\Gamma$. In particular, $R_\Gamma$ is a reduction system for $\kk Q'_\Gamma$ and satisfies the condition ($\diamond$) for $I'_\Gamma$, see  \cite[Lemma 2.10]{CS} or \cite[Section 3.1.1]{BW}.
	\end{proof}
	
	\begin{example}
		Consider the Brauer graph $(\Gamma,\m)$ where $\Gamma$ is given by:
		\begin{center}
			
			\tikzset{every picture/.style={line width=0.75pt}} 
			

		\end{center}
		and the multiplicity function $\m$ is defined by $\m(v_1)=\m(w)=1$, $\m(v_2)=2$. Then the associated quiver $Q_\Gamma$ is given by
		$$Q_\Gamma= \begin{tikzcd}
			1 \arrow["\alpha"', loop, distance=2em, in=215, out=145] \arrow[r, "\delta", shift left] & 2 \arrow["\beta"', loop, distance=2em, in=35, out=325] \arrow[l, "\gamma", shift left]
		\end{tikzcd}$$
		and $$I_\Gamma=\langle \alpha-\gamma\delta,\;\beta^2-\delta\gamma,\;\beta^3,\;\alpha^2,\;\alpha\gamma,\;\delta\alpha,\;\beta\delta,\;\gamma\beta\rangle.$$
		
		Note that $\Gamma$ is bipartite by choosing $V_1=\{w\}$ and $V_2=\{v_1,v_2\}$. Since there are no truncated vertices in $V_1$, we obtain that the corresponding quiver $Q_\Gamma'=Q_\Gamma$ and $I_\Gamma'=I_\Gamma$. In this case, the reduction system $R_\Gamma$ is given by
		$$\{(\gamma\delta,\alpha),\;(\delta\gamma,\beta^2),\;(\beta^3,0),\;(\alpha^2,0),\;(\alpha\gamma,0),\;(\delta\alpha,0),\;(\beta\delta,0),\;(\gamma\beta,0)\}.$$
		
		We may also choose $V_1=\{v_1,v_2\}$ and $V_2=\{w\}$. In this case, the vertex $v_1$ in $V_1$ is truncated so the loop $\alpha$ is deleted and the corresponding quiver $Q_\Gamma'$ is given by 
		$$Q_{\Gamma}' = \begin{tikzcd}
			1 \arrow[r, "\delta", shift left] & 2 \arrow["\beta"', loop, distance=2em, in=35, out=325] \arrow[l, "\gamma", shift left]
		\end{tikzcd}$$
		and the ideal $I_\Gamma'$ is given by  $$I_\Gamma'=\langle \beta^2-\delta\gamma,\;\gamma\delta\gamma,\;\delta\gamma\delta,\;\beta\delta,\;\gamma\beta\rangle.$$
		Thus the reduction system $R_\Gamma$ is given by 
		$$\{(\beta^2,\delta\gamma),\;(\gamma\delta\gamma,0),\;(\delta\gamma\delta,0),\;(\beta\delta,0),\;(\gamma\beta,0)\}.$$
	\end{example}

	At the end of this subsection, we would like to mention that, as an application of the Gr\"{o}bner basis constructed in \cite{LX},  we can calculate the dimension of any Brauer graph algebra (not necessarily bipartite). This result may be known to experts. Indeed, \cite[Proposition 3.13]{GS} provides a dimension formula for any Brauer configuration algebra, which is a generalization of Brauer graph algebra, using an approach that does not employ Gr\"{o}bner basis.
	
	\begin{proposition}\label{dim}
		Let $B_\Gamma$ be a Brauer graph algebra with associated Brauer graph $(\Gamma,\m)$. Then $$\dim_\kk(B_\Gamma)=\sum_{v\in V(\Gamma)}\m(v)\mathrm{val}(v)^2.$$
	\end{proposition}
	\begin{proof}
		We adopt the notations from \cite{LX}. By \cite[Proposition 4.3]{LX}, there is a Gr\"{o}bner basis $R_1\cup R_2\cup R_3$ with respect to some admissible well-order on $(Q_\Gamma)_{\geq 0}$ for $B_\Gamma$, where $R_1, R_2$, and $R_3$ consists of the defining relations of type I, II, and III respectively as in \cite[Subsection 4.1]{LX} (see Remark~\ref{remark:truncated} for a short review). It is known that the set of irreducible paths forms a basis of $B_\Gamma$. By considering the relations appeared in $R_3$, it is enough to consider the subpaths of any cycle of the form $C_\alpha^{\m(C_\alpha)}$ when we count the number of irreducible paths.  Note that every proper subpath of the cycle $C_\alpha^{\m(C_\alpha)}$ does not contain any tips and thus is irreducible. There are the number 
		\[
		|E(\Gamma)| + \sum_{v\in V}\mathrm{val}(v)(\m(v)\mathrm{val}(v)-1)
		\]
		of such irreducible paths. 
		For each edge $\bar h$ in $\Gamma$ either $C_{\alpha_h}^{\m(s(h))}$ or $ C_{\alpha_{\iota(h)}}^{\m(s(\iota(h)))}$ is irreducible. The number of such irreducible paths is equal to the number $|E(\Gamma)|$ of edges. As a result, we have the following 			
		\begin{align*}
			\dim_\kk(B_\Gamma) & =|E(\Gamma)|+\sum_{v\in V}\mathrm{val}(v)(\m(v)\mathrm{val}(v)-1)+|E(\Gamma)|\\
			& =\sum_{v\in V}\m(v)\mathrm{val}(v)^2-\sum_{v\in V}\mathrm{val}(v)+2|E(\Gamma)|\\
			&=\sum_{v\in V}\m(v)\mathrm{val}(v)^2-2|E(\Gamma)|+2|E(\Gamma)|\\
			&=\sum_{v\in V}\m(v)\mathrm{val}(v)^2,
		\end{align*}
		where the third equality follows since $\sum_{v\in V}\mathrm{val}(v) = 2 |E(\Gamma)|$.		
	\end{proof}
	
	\begin{remark}
		In Section \ref{subsection:typeA}, we show that each bipartite Brauer graph algebra $B_\Gamma$ admits a natural deformation $B_\Gamma^{(A)}$ which is isomorphic to the matrix algebra $\prod_{v\in V(\Gamma)}\mathrm{M}_{\mathrm{val}(v)}(\kk)^{\m(v)}$. Clearly the dimension of the latter is also given by $\sum_{v\in V(\Gamma)}\m(v)\mathrm{val}(v)^2$. 
	\end{remark}

	\subsection{General forms of $2$-cochains}\label{sec:differential}

	As is well-known, Hochschild cohomology groups are important invariants of associative algebras, preserved under derived equivalences. Their computation relies on the construction of a  projective bimodule resolution of the given algebra. Here we review a method for computing Hochschild cohomology of a quiver algebra via the projective resolution defined in \cite{CS}. However, for the purpose to relating our computation with the deformation theory, we will adopt the notations from \cite{BW} for this resolution.
	
	For each quiver algebra $A=\kk Q/I$, if we fix a reduction system $R$ satisfying the condition ($\diamond$) for $I$, then there is a two-sided projective resolution $P_\bullet$ of $A$ which is given by Chouhy and Solotar in \cite[Section 4]{CS} (see also \cite[Section 4]{BW}). Applying $\ho_{A^e}(-, A)$ to $P_\bullet,$ we obtain a cochain complex
	$$0\rightarrow \ho_{\kk Q_0^e}(\kk Q_0, A)\stackrel{\partial^0}{\rightarrow} \ho_{\kk Q_0^e}(\kk Q_1, A)\stackrel{\partial^1}{\rightarrow} $$ 
$$\ho_{\kk Q_0^e}(\kk S_2, A)\stackrel{\partial^2}{\rightarrow} \ho_{\kk Q_0^e}(\kk S_3, A)\rightarrow\cdots$$
	computing the Hochschild cohomology $\hh^\bullet(A).$ 
	Here,  $S_2:=S$ is the set of tips of $R$ and 
	$$S_3:=\{\text{(left) $1$-ambiguities with respect to $S$} \}.$$
	A path $uvw$ is called a {\it (left) $1$-ambiguity} with respect to $S$ if
	\begin{itemize}
	\item $u\in Q_1$ and $v,w$ are irreducible paths in $\mathrm{Irr}_S$;
	\item $uv$ (resp. $vw$) is reducible but $ud$ (resp. $vd'$) is irreducible for any proper (left) divisor $d$ (resp. $d'$) of $v$ (resp. $w$).
	\end{itemize}

So $\hh^2(A)$ can be represented by the non-trivial cocycles in $\ho_{\kk Q_0^e}(\kk S, A)$. In this paper, we often denote $\psi$ as an element of $\ho_{\kk Q_0^e}(\kk Q_1, A)$, and $\tilde{\varphi}$ (here we use $\tilde{\varphi}$ in order to avoid the confusion with the symbol $\varphi$ appeared in the reduction system) as an element of $\ho_{\kk Q_0^e}(\kk S, A)$. Note that the set $\irr_S$ of all irreducible paths forms a basis of $A$ through the natural projection $\pi: \kk Q \twoheadrightarrow \kk Q/I$.
	
	To be more explicit, the differentials in low degrees are explicitly given by:
	\begin{itemize}
		\item For each $\phi\in \mathrm{Hom}_{\kk Q_0^e}(\kk Q_0,{A})$ and each arrow $\alpha\in Q_1$, the differential $\partial^0$ is defined by: $$\partial^0(\phi)(\alpha)=\pi(\alpha\cdot \phi(s(\alpha)))-\pi(\phi(t(\alpha))\cdot\alpha);$$
		\item For each $s \in S$, let $\mathrm{Supp}(s - \varphi_s)$ denote the set of paths appearing in  $s - \varphi_s$, and for each path $p \in \mathrm{Supp}(s - \varphi_s)$, let $c(p) \in \kk$ be its coefficient in $s - \varphi_s$. That is, $$s-\varphi_s = \sum_{p\in \mathrm{Supp}(s - \varphi_s)} c(p) p.$$  For each $\psi \in \ho_{\kk Q_0^e}(\kk Q_1, A)$ and $s \in S$, the differential $\partial^1$ is defined by: $$\partial^1(\psi)(s)=\sum_{\alpha_{m}\cdots\alpha_{1}\in \mathrm{Supp}(s-\varphi_s)}\sum_{i=1}^{m}c(\alpha_{m}\cdots\alpha_{1})\cdot\pi(\alpha_{m}\cdots\alpha_{i+1}\psi(\alpha_{i})\alpha_{i-1}\cdots\alpha_{1});$$
		(For the above two differentials, see also in \cite[Proposition 3.7]{LX}.)
        \item Let $\partial_3: A \otimes_{\kk Q_0} \kk S_3 \otimes_{\kk Q_0} A \to A \otimes_{\kk Q_0} \kk S \otimes_{\kk Q_0} A$ be the differential in the two-sided projective resolution of $A$ from \cite{CS} (given implicitly in \cite[Theorems 4.1 and 4.2]{CS} and explicitly via homotopy deformation retract in \cite[Lemma 5.15]{BW}). For each $$\tilde\varphi \in \ho_{\kk Q_0^e}(\kk S, A) \cong \ho_{A^e}(A \otimes_{\kk Q_0} \kk S \otimes_{\kk Q_0} A, A)$$ and each $uvw \in S_3$, the differential $\partial^2$ is defined by:
$$\partial^2(\tilde\varphi)(uvw)=\tilde\varphi(\partial_3(1\otimes uvw\otimes 1)).$$
In fact, a combinatorial description of the kernel of $\partial^2$ is given in \cite[Section~7.A]{BW}, which is the main tool for us to compute the $2$-cocycles in $\hh^2(B_\Gamma)$, see Appendix \ref{appendix:2-cocycle} and the proof of Theorem \ref{thm:infinitesimaldeformation}. 
	\end{itemize}

	{\it In this section, we will use the reduction system $R_\Gamma$ constructed in Section \ref{subsection:reductionsystem} to give a collection of explicit cocycles which form a basis of $\hh^2(B_\Gamma)$ for any bipartite Brauer graph algebra $B_\Gamma.$}
	
	Let $\Gamma=\Gamma[V_1,V_2]$ be a bipartite Brauer graph and $B_\Gamma=\kk Q'_\Gamma/I'_\Gamma$ be the corresponding Brauer graph algebra as in Subsection \ref{subsection:reductionsystem}. For this moment, we assume that  $\Gamma$ is not the Brauer graph with two vertices with the multiplicities $m$ and $1$ and a single edge, so that the corresponding Brauer graph algebra is not isomorphic to the local Brauer tree algebra $\Bbbk[x]/\langle x^{m+1}\rangle $.  See Example \ref{exam:localbrauer} for the discussion on this special case. 
	
	For simplicity, let us denote $Q'=Q_\Gamma' $ and $I'=I_\Gamma'$.  First of all, we give an explicit description about the space $\ho_{\kk (Q')_0^e}(\kk S, B_\Gamma)$ with $S=\{s\;|\; (s,\varphi_s)\in R_\Gamma\}$. Recall that the set $\irr_S$ of all irreducible paths forms a basis of $B_\Gamma$, and according to the discussion in the proof of Proposition \ref{prop:reductionsystem}, each element of $\irr_S$ must be a subpath of some cycle of the form $C_\alpha^{\m(C_\alpha)}$.
	
Let $\tphi\in\ho_{\kk (Q')_0^e}(\kk S, B_\Gamma)$. The general form of $\tphi$ is described as follows, which depends on the elements of three types  $(a), (b), (c)$ in $R_\Gamma$, see the paragraph right below Lemma  \ref{iso}. Note that in the following when we view a path $p$ in the quiver $Q$ as an element of $A$, we always mean that it represents the image of $p$ in $\kk \irr_S$ through the canonical projection $\kk Q\rightarrow A$.
	
	\subsubsection{Let $(s,\varphi_s)$ be an element of type (a) in $R_\Gamma$}\label{sec:cochain(a)}
	\begin{enumerate}[(1)]
		\item 	Let $v\in V_1$ be not truncated (i.e.\ $\m(v)\neq 1$ and $\mathrm{val}(v)\neq 1$). For any $w\in V_2$  connected with $v$ through an edge $\bar h$. Denote  $\m(v)=m$ and  $\m(w)=n$. Let $e$ be the idempotent corresponding to $\bar h $. In this case, we have $(C_\alpha^m,C_\beta^n) \in R_\Gamma$, that is $s=C_\alpha^m$, $\varphi_s=C_\beta^n$.
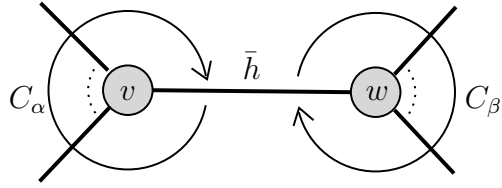
		\begin{center}
			
			\tikzset{every picture/.style={line width=0.75pt}} 
			

		\end{center}
		Then the general form of $\tphi_s:=\tphi(s)$ is given as $$\tphi_s=\lambda^{(s)}e+\sum_{i=1}^{m-1}\mu_i^{(s)}C_\alpha^i+\sum_{j=1}^{n-1}\varepsilon_j^{(s)}C_\beta^j+\kappa^{(s)}C_\beta^n,$$
		where $\lambda^{(s)}, \mu_i^{(s)}, \varepsilon_j^{(s)}, \kappa^{(s)} \in \Bbbk.$ In other words, the irreducible paths parallel to $s=C_\alpha^m$ are exactly  $e,C_\alpha^i$ and $ C_\beta^j$ for $1\leq i \leq m-1$ and $1\leq j \leq n$.
	\end{enumerate}

	\subsubsection{Let $(s,\varphi_s)$ be an element of type (b) in $R_\Gamma$}
	\begin{enumerate}[(1)]
		\setcounter{enumi}{1}
		\item Let $w\in V_2$.  Let $h_1\neq h_2$ be any two half-edges such that $s(h_1)= s(h_2) = w$ and $s(\iota(h_1))\neq s(\iota(h_2))$.  Denote $\m(w)=n$. In this case, we have $(C_\beta^n \beta, 0) \in R_\Gamma$, namely  $s=C_\beta^n \beta$, $\varphi_s=0$.
		
		
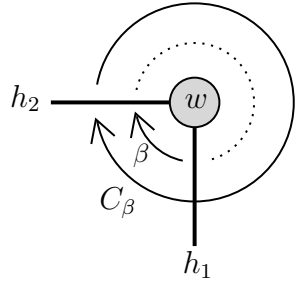
\begin{center}

			\tikzset{every picture/.style={line width=0.75pt}} 
			
	
		\end{center}
		Then the general form of $\tphi_s$ is given as 
		$$\tphi_s=\lambda^{(s)}\beta+\sum_{i=1}^{n-1}\varepsilon_i^{(s)}C_\beta^i \beta,$$
		where $\lambda^{(s)}, \varepsilon_i^{(s)} \in \Bbbk.$
		
		\item Let $v\in V_1$ and $w\in V_2$. Let $h_1\neq h_2$ be any two half-edges such that $s(h_1)= s(h_2) = w$ and $s(\iota(h_1))=s(\iota(h_2))=v$. Denote $\m(v)=m$ and $\m(w)=n$. In this case, we have  $(C_\beta^n \beta,0)\in R_\Gamma$, namely $s=C_\beta^n \beta$, $\varphi_s=0$.
		\begin{center}
			\tikzset{every picture/.style={line width=0.75pt}} 
			

			
		\end{center}
		Then the general form of $\tphi_s$ is given as  $$\tphi_s=\lambda^{(s)}\beta+\sum_{i=1}^{n-1}\varepsilon_i^{(s)}C_\beta^i\beta+\sum_{j=0}^{m-1}\nu_j^{(s)}C_\alpha^jp_\alpha,$$
		where $\lambda^{(s)}, \varepsilon_i^{(s)}, \nu_j^{(s)} \in \Bbbk.$
		
		\item Let $w\in V_2$ with $\mathrm{val}(w)=1$ and $v\in V_1$. Denote $\m(w)=n$ and $\m(v)=m$. Denote the cycle starting from $h$ by $C_\alpha$. In this case, we have $(\beta^{n+1},0)\in R_\Gamma$, namely $s=\beta^{n+1}$, $\varphi_s=0$.
		\begin{center}
			
			\tikzset{every picture/.style={line width=0.75pt}} 
			

		\end{center}
		Then the general form of $\tphi_s$ is given as $$\tphi_s=\lambda^{(s)}\beta+\sum_{i=1}^{n-1}\varepsilon_i^{(s)}\beta^{i+1}+\sum_{j=0}^{m-1}\nu_j^{(s)}C_\alpha^j.$$
		where $\lambda^{(s)}, \varepsilon_i^{(s)}, \nu_j^{(s)} \in \Bbbk.$
	\end{enumerate}
	
	\subsubsection{Let $(s,\varphi_s)$ be an element of type (c) in $R_\Gamma$}\label{sec:cochain(c)}
	\begin{enumerate}[(1)]
		\setcounter{enumi}{4}
		\item Let $(\alpha\beta, 0) \in R_\Gamma$. Suppose that $\alpha$ ends at the half-edge $h_1$ and $\beta$ starts from the half-edge $h_2$ such that $h_1\neq \iota(h_2)$. 
		
		\begin{center}
			\tikzset{every picture/.style={line width=0.75pt}} 
			

			
		\end{center}
		Since $\Gamma$ is bipartite,  $s(\iota(h_1))\neq s(\iota(h_2))$ it follows that $$\tphi_{\alpha \beta}=0.$$
		Namely, in this case, there are no irreducible paths parallel to $\alpha \beta$.
		
		\item Let $v,w\in V$ be different vertices connected by some edges in $\Gamma$.  Let $h_1\neq h_2$ be any two half-edges such that $s(h_1)= s(h_2) = w$ and $s(\iota(h_1))= s(\iota(h_2))=v$. Denote $\m(v)=m$ and $\m(w)=n$. Let $e$ be the idempotent corresponding to $\bar{h_2}$. Denote the cycle starting from $\iota(h_2)$ by $C_{\sigma(\alpha)}$ and the cycle starting from $h_2$ by $C_\beta$. In this case, we have $(\beta\alpha,0)\in R_\Gamma$, that is $s=\beta\alpha$, $\varphi_s=0$.
		
		
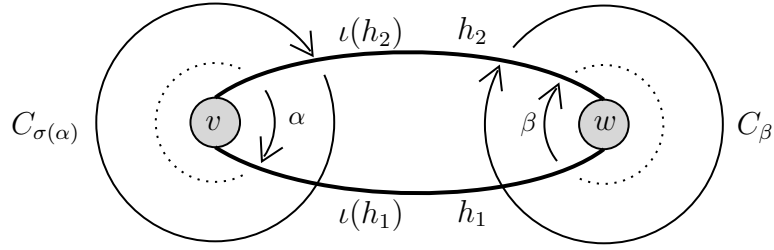
\begin{center}

			\tikzset{every picture/.style={line width=0.75pt}} 
			

			
		\end{center}
		Then the general form of $\tphi_s$ is given as $$\tphi_s=\lambda^{(s)}e+\sum_{i=1}^{m-1}\xi_i^{(s)}C_{\sigma(\alpha)}^i+\sum_{j=1}^{n-1}\zeta_j^{(s)}C_\beta^j+\kappa^{(s)}C_\beta^n,$$
		where $\lambda^{(s)}, \xi_i^{(s)}, \zeta_j^{(s)}, \kappa^{(s)} \in \Bbbk.$
		
		\item Let $v,w\in V$ be different vertices connected by some edges in $\Gamma$.  Let $h_1\neq h_2$ be any two half-edges such that $s(h_1)= s(h_2) = w$ and $s(\iota(h_1))= s(\iota(h_2))=v$. Denote $\m(v)=m$ and $\m(w)=n$. In order to make this case different from (6), if $p_{\alpha'}, p_{\beta}\in (Q_\Gamma')_1$, then we assume $\bar{h_1}\neq \bar{h_2}$. Denote the cycle starting from $\iota(h_2)$ by $C_{\alpha'}$ and the cycle starting from $h_2$ by $C_{\beta'}$. In this case, we have $(\beta'\alpha,0)\in R_\Gamma$, that is $s=\beta'\alpha$, $\varphi_s=0$.
		
		
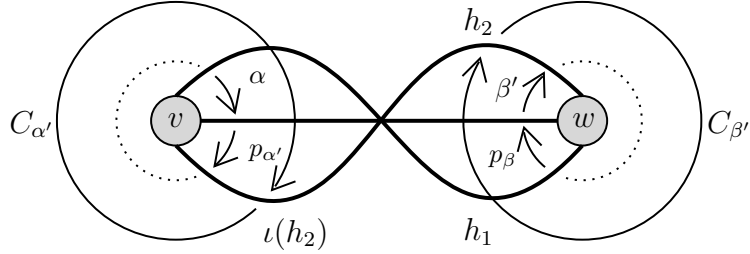
\begin{center}

			\tikzset{every picture/.style={line width=0.75pt}} 
			

			
		\end{center}
		Then the general form of $\tphi_s$ is given as $$\tphi_s=\sum_{i=0}^{m-1}\xi_i^{(s)}C_{\alpha'}^ip_{\alpha'}\alpha+\sum_{j=0}^{n-1}\zeta_j^{(s)}C_{\beta'}^j\beta'p_{\beta},$$
		where $\xi_i^{(s)}, \zeta_j^{(s)} \in \Bbbk.$
		
		\item Let $v,w\in V$ be different vertices connected by some edges in $\Gamma$.  Let $h_1\neq h_2$ be any two half-edges such that $s(h_1)= s(h_2) = w$ and $s(\iota(h_1))=v\neq s(\iota(h_2))$. Denote $\m(w)=n$. Let the cycle starting from $h_2$ by $C_{\beta'}$. In this case, we have $(\beta'\alpha,0)\in R_\Gamma$, that is $s=\beta'\alpha$, $\varphi_s=0$.
		
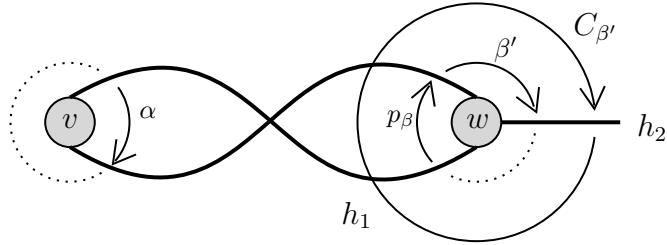
\begin{center}

			\tikzset{every picture/.style={line width=0.75pt}} 
			

			
		\end{center}
		Then the general form of $\tphi_s$ is given as $$\tphi_s=\sum_{j=0}^{n-1}\zeta_j^{(s)}C_{\beta'}^j\beta'p_\beta,$$
		where $ \zeta_j^{(s)} \in \Bbbk.$
		
		\item Let $v,w\in V$ be different vertices connected by some edges in $\Gamma$.  Let $h_1\neq h_2$ be any two half-edges such that $s(h_1)= s(h_2) = v$ and $s(\iota(h_1))\neq w = s(\iota(h_2))$. Denote $\m(v)=m$. Denote the cycle starting from $h_2$ by $C_{\alpha}$. In this case, we have $(\beta\alpha',0)\in R_\Gamma$, that is $s=\beta\alpha'$, $\varphi_s=0$.
		
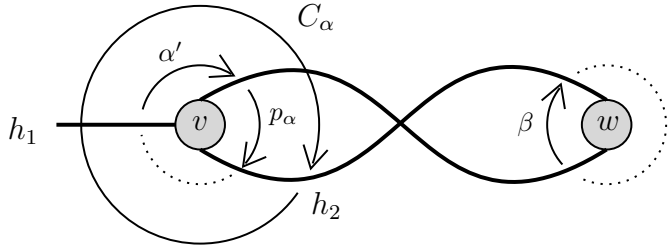
\begin{center}

			\tikzset{every picture/.style={line width=0.75pt}} 
			

			
		\end{center}
		Then the general form of $\tphi_s$ is given as $$\tphi_s=\sum_{i=0}^{m-1}\xi_i^{(s)}C_{\alpha}^i p_\alpha\alpha',$$
		where $\xi_i^{(s)} \in \Bbbk.$
		
		\item Let $v,w\in V$ be different vertices connected by a single edge in $\Gamma$ with $\mathrm{val}(v)=\mathrm{val}(w)=1$. Denote  $\m(v)=m$ and $\m(w)=n$. Let $e$ be the idempotent corresponding to this edge.  In this case, we have $(\beta\alpha,0)\in R_\Gamma$, namely $s=\beta\alpha$, $\varphi_s=0$.
		\begin{center}
			
			\tikzset{every picture/.style={line width=0.75pt}} 
			
			%
		\end{center}
		Then the general form of $\tphi_s$ is given as $$\tphi_s=\lambda^{(s)}e+\sum_{i=1}^{m-1}\xi_i^{(s)}\alpha^i+\sum_{j=1}^{n-1}\zeta_j^{(s)}\beta^{j}+\kappa^{(s)}\beta^n,$$
		where $\lambda^{(s)}, \xi_i^{(s)}, \zeta_j^{(s)}, \kappa^{(s)} \in \Bbbk.$
		
		\item Let $v,w\in V$ be different vertices connected by a single edge in $\Gamma$ with $\mathrm{val}(v)\neq 1=\mathrm{val}(w)$. Denote $\m(v)=m$. In this case, we have $(\beta\alpha,0)\in R_\Gamma$, that is $s=\beta\alpha$, $\varphi_s=0$.
		\begin{center}

			\tikzset{every picture/.style={line width=0.75pt}} 
			


		\end{center}
		Then the general form of $\tphi_s$ is given as $$\tphi_s=\sum_{i=0}^{m-1}\xi_i^{(s)}C_\alpha^i\alpha.$$
		where $\xi_i^{(s)} \in \Bbbk.$
		
		\item Let $v,w\in V$ be different vertices connected by a single edge in $\Gamma$ with $\mathrm{val}(w)\neq 1=\mathrm{val}(v)$. Denote $\m(w)=n$. In this case, we have $(\beta\alpha,0)\in R_\Gamma$, that is $s=\beta\alpha$, $\varphi_s=0$.
		\begin{center}

			\tikzset{every picture/.style={line width=0.75pt}} 
			
			%

		\end{center}
		Then the general form of $\tphi_s$ is given as $$\tphi_s=\sum_{j=0}^{n-1}\zeta_j^{(s)}C_\beta^{j}\beta,$$
		where $\zeta_j^{(s)} \in \Bbbk.$
	\end{enumerate}

\begin{remark}
	It is not hard to see that each term appeared in the above sums (for example, the term $C_\alpha^i$ in the sum of (1), the term $C_\beta^{j}\beta$ in the sum of (12), etc.) constructed in Subsections \ref{sec:cochain(a)}--\ref{sec:cochain(c)} contributes one dimension for the $\kk$-space $\ho_{\kk (Q')_0^e}(\kk S, B_\Gamma)$ of 2-cochains (of course, in concrete examples, only part of the terms will appear). In Appendix, we will show that the non-trivial cocycles in the second Hochschild cohomology group $\hh^2(B_\Gamma)$ only involve those terms with the coefficients $\lambda,\mu,\varepsilon,\kappa$.
	\end{remark}

	\subsection{Hochschild cohomology in degree $2$}
	
	Based on the general form of $2$-cochains above, we may obtain an explicit basis of  $\hh^2(B_\Gamma)$ of a Brauer graph algebra $B_\Gamma$ as follows.

	For convenience, we recall that a {\it spanning tree} of a graph $\Gamma$ is a subgraph whose vertex set coincides with the entire vertex set of $\Gamma$ and which is itself a tree \cite{BM}.

	\begin{proposition}\label{prop:cocycle}
		Let $B_\Gamma$ be a bipartite Brauer graph algebra with associated Brauer graph $(\Gamma,\m)$. If $B_\Gamma$ is not a local Brauer tree algebra, then $\hh^2(B_\Gamma)$ admits a $\Bbbk$-linear basis represented by the following forms of standard cocycles $\tphi\in\ho_{\kk (Q')_0^e}(\kk S, B_\Gamma)$.
		\begin{enumerate}[(A)]
			\item  Fix a $\lambda\in \kk^*$. Define a cocycle $\tphi\in\ho_{\kk (Q')_0^e}(\kk S, B_\Gamma)$ as
			\begin{align*}
				\tphi_{C_\alpha^{\m(v)}} & =\lambda e,\\
				\tphi_{C_\beta^{\m(w)}\beta} &=-\lambda \beta,
			\end{align*}
			for all edges $\bar h\in E(\Gamma)$ connecting $v\in V_1$ to $w\in V_2$.
			\begin{figure}[H]
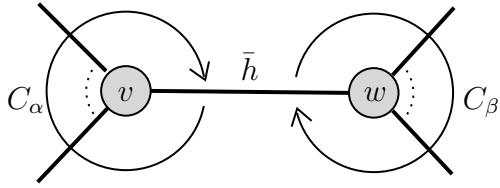

				\centering
				
				\tikzset{every picture/.style={line width=0.75pt}} 
				

				\caption{Cycles associated with the given edge $\bar h$.}
				\label{fig:(A)}
			\end{figure}
			Recall that in this case we have $$(C_\alpha^{\m(v)},  C_\beta^{\m(w)}), \ (C_\beta^{\m(w)}\beta, 0) \in R_\Gamma.$$  Here, $\lambda \in \kk$ and $e$ is the idempotent corresponding to $\bar h $. We define $\tphi_s = 0$ for all the other $s\in S.$

			\item Let $v\in V$ be a vertex in $\Gamma$ with a cycle $C_{\alpha_{h_1}}=\alpha_{h_1}\alpha_{h_2}\cdots\alpha_{h_l}$ around $v$. That is, $h_j^-=h_{j+1}$, for all $1\leq j\leq l$. We have two cases.
			\begin{itemize}
				\item Suppose $v\in V_1$ and fix $1\leq i\leq \m(v)-1$. Fix a $\mu\in \kk^*$. Then there exists a cocycle $\tphi^i\in\ho_{\kk (Q')_0^e}(\kk S, B_\Gamma)$  given by
				$$\tphi^i_{C_{\alpha_{h_j}}^{\m(v)}}=\mu C_{\alpha_{h_j}}^i,\ \text{for $1\leq j\leq l$.} $$
				We define $\tphi^i_s = 0$ for all the other $s\in S.$ In this case, there is the number $\m(v)-1$ of cocycles $\tphi^i$. 
				\item  Suppose $v\in V_2$ and fix $1\leq i\leq \m(v)-1$. Fix a $\varepsilon\in \kk^*$. Then there exists a cocycle $\tphi^i\in\ho_{\kk (Q')_0^e}(\kk S, B_\Gamma)$ given by
				\begin{align*}
					\tphi^i_{C_{\sigma(\alpha_{\iota(h_j)})}^{\m(s(\iota(h_j)))}} & =\varepsilon C_{\alpha_{h_{j-1}}}^i,\\
					\tphi^i_{C_{\alpha_{h_j}}^{\m(v)}{\alpha_{h_j}}} & =-\varepsilon C_{\alpha_{h_j}}^i \alpha_{h_j},
				\end{align*}
				for $1\leq j\leq l$.  We define $\tphi^i_s = 0$ for all the other $s\in S.$  In this case, there is the number $\m(v)-1$ of cocycles $\tphi^i$. (For arrows and cycles in this case, see for example in Figure \ref{fig:arrows-in-cycle})
			\end{itemize}
			
			\begin{figure}[ht]
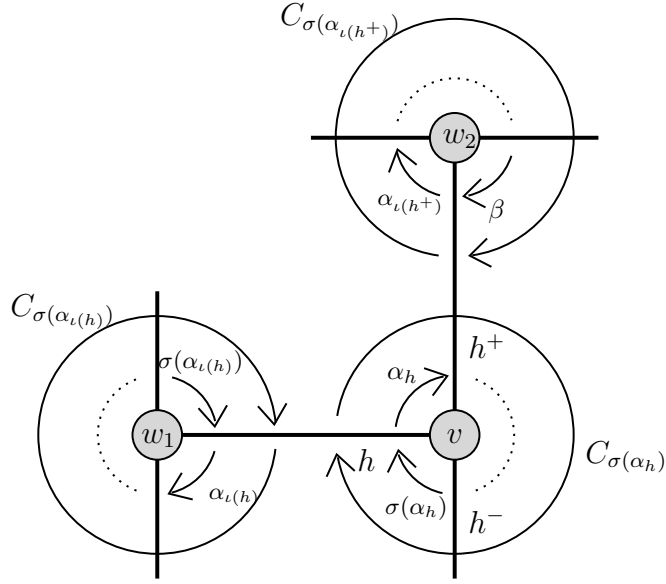

				\centering

				\tikzset{every picture/.style={line width=0.75pt}} 
				


				\caption{Arrows and cycles in bipartite Brauer graph $\Gamma$ where $w_1=s(\iota(h))$, $w_2=s(\iota(h^+))$ and $\beta=\sigma(\alpha_{\iota(h^+)})$.}
				\label{fig:arrows-in-cycle}
			\end{figure}
			
			\item Fix a spanning tree $T$ of $\Gamma$. For each edge $\bar h \in E(\Gamma)$ which is not in $T$ connecting a vertex $v\in V_1$ to  $w\in V_2$, we assume that $C_{\alpha}$ (resp.\ $C_\beta$) is the cycle starting from $\bar h$ around $v$ (resp.\ $w$). Fix a $\kappa\in \kk^*$. There exists a cocycle which is given by 
			$$\tphi_{C_{\alpha}^{\m(v)}}=\kappa C_{\beta}^{\m(w)},$$
			where $\kappa\in\kk$ and for all the other $s\in S,$ we define $\tphi^j_s =0$. (For arrows and cycles in this case, see for example in Figure \ref{fig:arrows-in-cycle}.) Note that the number of such edges $\bar h$ is $|E|-|V|+1$, see the explain at the beginning of Subsection \ref{subsec:typeC-4.2}.

			\item Let $v\in V_1$ and $w\in V_2.$ Let $h_1,h_2\in H_w$ be such that  $\iota(h_1), \iota(h_2)\in H_v$, $h_1^+=h_2$ and $\iota(h_1)^-= \iota(h_2)$. Note that here $h_1$ may be equal to $h_2$.
			
			\begin{figure}[H]
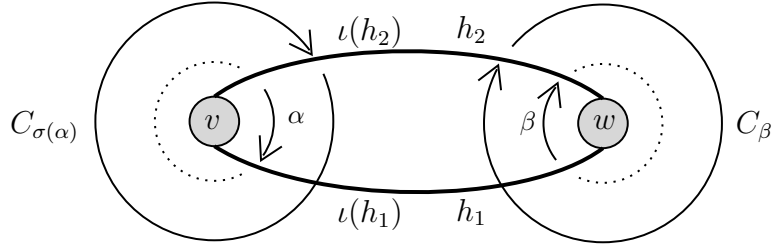

				\centering
				\tikzset{every picture/.style={line width=0.75pt}} 
				

				\caption{The subgraph of $\Gamma$ which can induce deformations of type (D).}
				\label{fig:(D)type}
			\end{figure}
			Then there exist two cocycles $\tphi^1$ and $\tphi^2$ which is given by
			\begin{align*}
				\tphi_{\alpha\beta}^1&=\lambda e_1,\\
				\tphi_{\beta\alpha}^1&=\lambda e_2,\\
				\tphi_{C_\beta^{\m(w)}\beta}^1&=\lambda \alpha^*,
			\end{align*}
			with $\lambda\in\kk$ (for the definition of $\alpha^*$, see Subsection \ref{subsection:def-of-BGA}) and 
			$$\tphi_{\beta\alpha}^2=\kappa C_\beta^{\m(w)},$$
			with $\kappa\in\kk$.  Here, $e_1$ (resp.\ $e_2$) is the idempotent corresponding to $\bar h_1 $ (resp.\ $\bar h_2$). We define $\tphi_s = 0$ for all the other $s\in S.$ 
		\end{enumerate}
	\end{proposition}
	\begin{proof}
		This follows from the computations in Appendix, see Remarks \ref{rmk:cocycle} and \ref{rmk:coboundaries}.
	\end{proof}	
	
	Define the following set: $$\mathrm{S_{2\text{-}cyc}}=\{\alpha\beta\;|\;\text{$\alpha\beta$ is a cycle in $B_\Gamma$ with $\alpha\neq\beta$, $\alpha\beta,\beta\alpha\in I_\Gamma$}\}.$$
	Then as a corollary, we have the following result.	
	
	\begin{corollary}\label{dimhh2}
		Let $B_\Gamma$ be a bipartite Brauer graph algebra corresponding to the bipartite Brauer graph $(\Gamma,\m)$. If $B_\Gamma$ is not a local Brauer tree algebra, then $$\dim_\kk\hh^2(B_\Gamma)=2+\sum_{v\in V(\Gamma)}(\m(v)-1)+|E|-|V|+|S_{2\text{-}cyc}|.$$
	\end{corollary}
	\begin{proof}
		Note that the numbers of cocycles of  type (A)-(D) in Proposition \ref{prop:cocycle} are respectively $$1, \sum_{v\in V(\Gamma)}(\m(v)-1), \ |E|-|V|+1,\  |S_{2\text{-}cyc}|.$$ By Proposition \ref{prop:cocycle} these cocycles form a basis of $\hh^2(B_\Gamma).$
	\end{proof}

\subsection{Deformations arising from the second Hochschild cohomology group}\label{sec:deformation}

		It is well-known that for any associative algebra $A$ with the multiplication map $\mu\colon A\otimes A\rightarrow A$, the elements of the second Hochschild cohomology group corresponds bijectively to the infinitesimal deformations (up to equivalences) of $A$, and any formal deformation of $A$ is equivalent to a special formal deformation (where the multiplication is given by $\mu_t=\mu+\mu_nt^n+\mu_{n+1}t^{n+1}+\cdots$), in which the first non-vanishing $\mu_n$ is a non-trivial Hochschild 2-cocycle (see \cite{Ger2}). Therefore, it is natural to understand the second Hochschild cohomology group through deformation theory. Furthermore, \cite[Theorem 7.1]{BW} establishes that formal deformations of quiver algebras are equivalent to formal deformations of reduction systems.

		For completeness, we recall the definition of formal deformations of reduction systems. Assume that $A=\Bbbk Q/I$ is a quiver algebra and $R$ is a reduction system for $\Bbbk Q$ satisfying condition ($\diamond$) for $I$. Denote by $\mathfrak{m}$ the maximal ideal $\langle t\rangle$ of $\Bbbk[[t]]$. Let 
		$${\widehat{\varphi}}\in\ho_{\kk Q_0^e}(\kk S, A)\widehat{\otimes}\mathfrak{m}$$
		so that $\tphi$ can be written as 
		$${\widehat{\varphi}}=\tphi_1t+\tphi_2t^2+\tphi_3t^3+\cdots$$
		with $\tphi_i\in\ho_{\kk Q_0^e}(\kk S, A)$ and let ${\widehat{\varphi}}^{(n)}=\tphi_1t+\cdots+\tphi_nt^n$ denote the image under tensoring by $-\otimes_{\Bbbk[[t]]}\Bbbk[t]/\langle t^{n+1}\rangle$.
		
		\begin{definition}
			A formal deformation of a reduction system $R=\{(s,\varphi_s)\}_{s\in S}$ for $\kk Q$ is given by 
		$$\widehat{R}_{\varphi+{\widehat{\varphi}}}=\{(s,\varphi_s+{\widehat{\varphi}}_s)\}_{s\in S},$$
		where ${\widehat{\varphi}}\in\ho_{\kk Q_0^e}(\kk S, A)\widehat{\otimes}\mathfrak{m}$ such that for each $n\geq 1$, the reduction system $$R_{\varphi+{\widehat{\varphi}}}^{(n)}=\{(s,\varphi_s+{\widehat{\varphi}}_s^{(n)})\}_{s\in S}$$ for $\Bbbk Q[t]/\langle t^{n+1}\rangle\simeq (\Bbbk[t]/\langle t^{n+1}\rangle)Q$ satisfies that every path in $Q$ is reduction-unique (or equivalently, the reduction system $R_{\varphi+{\widehat{\varphi}}}^{(n)}$ satisfies the condition ($\diamond$) for the ideal $\Bbbk[t]/\langle t^{n+1}\rangle\otimes_{\Bbbk[[t]]}I[[t]]$ of the path algebra $(\Bbbk[t]/\langle t^{n+1}\rangle)Q$ over the commutative ring $\Bbbk[t]/\langle t^{n+1}\rangle$).
		\end{definition}
		
The reason that $R^{(n)}_{\varphi+{\widehat{\varphi}}} =\{(s,\varphi_s+{\widehat{\varphi}}^{(n)}_s)\}_{s\in S}$ is a reduction system for $\Bbbk Q[t]/\langle t^{n+1}\rangle
\simeq (\Bbbk[t]/\langle t^{n+1}\rangle)Q$
is as follows. As in Definition~\ref{def:red-sys}, the set $S$ is tip--reduced, each $s$ is parallel to $\varphi_s+{\widehat{\varphi}}^{(n)}_s$, and every path occurring in $\varphi_s+{\widehat{\varphi}}^{(n)}_s$ is irreducible, since each ${\widehat{\varphi}}_s$ is a linear combination of irreducible paths in $A[[t]]$. Moreover, $\widehat{R}_{\varphi+{\widehat{\varphi}}}$ is a formal deformation of $R$ if and only if
\[
\Bbbk Q[[t]] \big/ \langle s-\varphi_s-{\widehat{\varphi}}_s\rangle^{\hat{}}_{s\in S}
\]
is a formal deformation of $A=\Bbbk Q/I$, where the underlying $\Bbbk[[t]]$--module is identified with $A[[t]]$ and $\langle s-\varphi_s-{\widehat{\varphi}}_s\rangle^{\hat{}}_{s\in S}$ denotes the two--sided ideal of $\Bbbk Q[[t]]$ generated by the elements $\{s-\varphi_s-{\widehat{\varphi}}_s\}_{s\in S}$, viewed as a formal
deformation of $I=\langle s-\varphi_s\rangle_{s\in S}$ in the sense of
\cite[Section~7.1]{BW}. Equivalently, this means that $(A[[t]],\mu_t)$ is an associative $\Bbbk[[t]]$--algebra, where the multiplication $\mu_t$ is induced from the quotient $\Bbbk Q[[t]]/\langle s-\varphi_s-{\widehat{\varphi}}_s\rangle^{\hat{}}_{s\in S}$. For instance, for each $s\in S$, the original multiplication satisfies $\mu(s,1)=\varphi_s$, while after deformation one has $\mu_t(s,1)=\varphi_s+{\widehat{\varphi}}_s$.

		In Theorem \ref{thm:infinitesimaldeformation}, by constructing formal deformations of the reduction system, we will prove that for a given bipartite Brauer graph algebra, every infinitesimal deformation corresponding to a standard $2$-cocycle in Proposition \ref{prop:cocycle} can be lifted to a formal deformation. Moreover, the algebraization of these formal deformations induces corresponding deformations of the Brauer graph algebra (that is, new associative algebra structures defined on the underlying vector space). These deformations can also be understood through the following algebraic variety.

		Recently, Green, Hille and Schroll \cite{GHS} constructed an algebraic variety for a given finite quiver $Q$, induced by the reduced Gr\"{o}bner bases in $\Bbbk Q$ (which may be viewed as special reduction systems) that share the same tip set. Each point on this variety can be regarded as a deformation of the reduction system, thereby corresponding to a deformation of the quiver algebra. 

		To be more specific, let $A=\Bbbk Q/I$ be a quiver algebra and $R=\{(s,\varphi_s)\}_{s\in S}$ be a reduction system for $\Bbbk Q$ satisfying condition ($\diamond$) for $I$. For each $\tphi\in\ho_{\kk Q_0^e}(\kk S, A)$, we can deform $R$ as follows:
		$$R_{\varphi+\tphi}:=\{(s,\varphi_s+\tphi_s)\}_{s\in S},$$
		and the corresponding associative algebra is given by
		$$A_{\varphi+\tphi}:=\Bbbk Q/ \langle s-\varphi_s-\tphi_s\rangle.$$
		Actually, $A$ and $A_{\varphi+\tphi}$ belong to the same algebraic variety defined in \cite{GHS} if and only if $R_{\varphi+\tphi}$ is also a reduction system for $\Bbbk Q$ satisfying condition ($\diamond$) for $\langle s-\varphi_s-\tphi_s\rangle$.
		
		For a given bipartite Brauer graph algebra $B_\Gamma$ with the reduction system $R_\Gamma$ defined in Subsection \ref{subsection:reductionsystem}, the following theorem shows that when the chosen $\tphi$ is a $2$-cocycle as described in Proposition \ref{prop:cocycle}, the resulting reduction system yields a deformation of $B_\Gamma$.
		
	\begin{theorem}\label{thm:infinitesimaldeformation}
		Let  $\Gamma=\Gamma[V_1,V_2]$ be a bipartite Brauer graph and $B_\Gamma$ be the associated Brauer graph algebra. If $B_\Gamma$ is not a local Brauer tree algebra, then for each standard $2$-cocycle $\tphi$ defined in Proposition \ref{prop:cocycle}, 
		$$\widehat{R}_{\varphi+\tphi}=\{(s,\varphi_s+\tphi_s t)\}_{s\in S}$$ is a formal deformation of $R_\Gamma$.
		
		In particular,  setting $t=1$, the formal deformations of $R_\Gamma$ induce deformations of the algebra $B_\Gamma$ by deforming the generating relations in $B_\Gamma$ (as defined in Subsection \ref{subsection:def-of-BGA}). More precisely, we have the following description.
		
				\begin{enumerate}[(A)]
			\item Recall that for each edge in $\Gamma$, we have the commutative relation $C_\alpha^{\m(C_\alpha)}=C_\beta^{\m(C_\beta)}$ with $C_\alpha$  a cycle around a vertex $v\in V_1$. There is a deformation parametrized by $\lambda$ given by replacing all the commutative relations $C_\alpha^{\m(C_\alpha)}=C_\beta^{\m(C_\beta)} $ with
			$$C_\alpha^{\m(C_\alpha)}=C_\beta^{\m(C_\beta)}-\lambda t(C_\alpha)$$ where $t(C_\alpha)$ is the idempotent corresponding to the terminus vertex of $C_\alpha$. Keep the other relations. 
			
			\item Fix a cycle $C_\alpha=\alpha_m\cdots\alpha_1$ (up to cyclic permutation) and let $i\in \{1,2,\cdots,\m(C_\alpha)-1\}$. Write $C_{\alpha_k}=  \alpha_{k} \dotsc \alpha_1\alpha_m \dotsc\alpha_{k-1}$. 
			We have the following deformations given by replacing every relation $C_{\alpha_k}^{\m(C_\alpha)}-C_{\beta_k}^{\m(C_{\beta_k})}=0$ ($1\leq k\leq m$) with
			$$C_{\alpha_k}^{\m(C_\alpha)}-C_{\beta_k}^{\m(C_{\beta_k})}+\mu_i^{(C_\alpha)}C_{\alpha_k}^i =0.$$ Keep the other relations. 
			
			\item Fix a spanning tree $T$ of $\Gamma$. For each edge not in $T$, there corresponds a commutative relation $C_{\alpha}^{m}-C_{\beta}^{n} =0$ where $m=\m(C_{\alpha})$, $n=\m(C_{\beta})$. In particular, we assume that $C_{\alpha}$ is a cycle induced by a vertex $v\in V_1$. This yields the following deformation of the relation:
			$$C_{\alpha}^{m}-C_{\beta}^{n}+\kappa^{(C_{\alpha})}C_{\beta}^{n}=0.$$ Keep the other relations and add $C_\gamma^{\m(C_\gamma)}\gamma=0$ for each cycle $C_\gamma$ around some vertex in $V_2$. 
			
			\item Let $\alpha\beta$ be a $2$-cycle in $ Q_\Gamma$, where $\alpha$ is an arrow around some vertex in $V_1$, $\beta$  is an arrow around some vertex in $V_2$. In fact, we have that both $\alpha\beta=0$ and $\beta\alpha=0$ are in $B_\Gamma$. Then we have the following deformations of these relations:
			\begin{itemize}
				\item replace both $\alpha\beta=0$ and $\beta\alpha=0$ with
			$$\alpha\beta+\lambda't(\alpha)=0,$$
			$$\beta\alpha+\lambda't(\beta)=0$$   respectively. Keep the other relations.
			\item  replace the relation $\beta\alpha=0$ with
			$$\beta\alpha+\kappa'C_\beta^{\m(C_\beta)}=0.$$  Keep the other relations.
			\end{itemize}			
		\end{enumerate}
	\end{theorem}
	
	\begin{proof}
		By \cite[Theorems 7.39 and 7.44]{BW}, it suffices to verify that for every path $uvw \in S_3$ with $uv, vw \in S$, the following associativity condition holds under the deformed multiplication: 
	\begin{equation}
(\pi(u)\star\pi(v))\star\pi(w)
=
\pi(u)\star(\pi(v)\star\pi(w)).
\tag{$\ast$}
\end{equation}
		Here, $\pi: \kk Q \twoheadrightarrow \kk Q/I$ is the natural projection onto the space of irreducible paths $\irr_S$, and $\star = \star^C_{\varphi+\tphi t}$ is the multiplication in the deformed algebra $\Bbbk Q[[t]] \big/ \langle s-\varphi_s-\tphi_s t\rangle^{\hat{}}_{s\in S}$, defined by multiplying paths in $\kk Q[[t]]$ and reducing via the reduction system $\widehat{R}_{\varphi+\tphi t}$. This reduction process means that (see \cite[Section 7.3]{BW}) for any path $p$ containing a subpath $s$ (i.e., $p = qsr$), we replace the occurrence of $s$ with $q(\varphi_s + \tphi_st)r$. The reduction process continues iteratively until the result is irreducible. The reduction used here is consistent with the one introduced at the beginning of Appendix~\ref{appendix:2-cocycle}. We emphasize that, in order to determine whether this deformation can be lifted, the equality~$(\ast)$ is required to hold in a formal sense, that is, the coefficients of $t^n$ on both sides must agree for all $n$. By contrast, when computing the $2$-cocycle in Appendix~\ref{appendix:2-cocycle}, it suffices for the equality~$(\ast)$ to hold modulo $t^2$.

		This verification is straightforward (in direct analogy to the verification provided in Appendix \ref{appendix:2-cocycle}). We illustrate this with a concrete example. Let $\tphi$ be the standard cocycle of type (A) defined in Proposition \ref{prop:cocycle} and consider the $1$-ambiguity $\beta\alpha\alpha^* = \beta C_\alpha^m$ from Case 8 in Appendix \ref{appendix:2-cocycle}, corresponding to $u=\beta$, $v=\alpha$, $w=\alpha^*$. We now verify the associativity condition for the deformed product:
		\begin{align*}(\pi(\beta)\star\pi(\alpha))\star\pi(\alpha^*)&=0\star\pi(\alpha^*)=0;\\\pi(\beta)\star(\pi(\alpha)\star\pi(\alpha^*))&=\pi(\beta)\star(C_{\sigma(\beta)}^n+\lambda et)=((\varphi_{\beta C_{{\sigma(\beta)}}^n}+\tphi_{\beta C_{{\sigma(\beta)}}^n} t)+\lambda \beta t)=((0-\lambda \beta t)+\lambda \beta t)=0,            \end{align*} where $e=t(C_\alpha)$ denotes the idempotent corresponding to the terminus vertex of $C_\alpha$.
		Then by \cite[Theorems 7.39 and 7.44]{BW}, this reduction system is a formal deformation of $R_\Gamma$. 
		
		By the isomorphism in Lemma \ref{iso}, each relation corresponding to $(C_\alpha^{\m(C_\alpha)}\alpha,0)$ in $R_\Gamma$ can be generated by other two types of relations in $R_\Gamma$. By Proposition \ref{prop:cocycle}, the coefficients which defined deformations of  $(C_\alpha^{\m(C_\alpha)}\alpha,0)$ can be induced by deformations of the other two types of relations in $R_\Gamma$.
	\end{proof}

	\begin{remark}
		According to \cite[Theorem 5.4]{MRRS} and \cite{BSW1}, all infinitesimal deformations of gentle algebras can be simultaneously lifted to formal deformations.
		
		For Brauer graph algebras (which are the trivial extensions of gentle algebras), our Theorem \ref{thm:infinitesimaldeformation} proves that a specific basis of infinitesimal deformations (namely, those corresponding to standard $2$-cocycles) for bipartite Brauer graph algebras can be lifted to formal deformations. However, not all infinitesimal deformations corresponding to arbitrary $2$-cocycles are liftable. For example, let $$\tphi=(\tphi^1+\tphi^2)t$$ where $\tphi^1$ and $\tphi^2$ are different standard $2$-cocyles of type (D) in Proposition \ref{prop:cocycle}. Then $\widehat{R}_{\varphi+\tphi}=\{(s,\varphi_s+\tphi_s)\}_{s\in S}$ is not a formal deformation of $R_\Gamma$: Consider the $1$-ambiguity $\alpha\beta\alpha$ from Case 5 in Appendix \ref{appendix:2-cocycle}, corresponding to $u=\beta$, $v=\alpha$, $w=\alpha^*$. We now verify the associativity condition for the deformed product:
		\begin{align*}(\pi(\alpha)\star\pi(\beta))\star\pi(\alpha)&=(\lambda t(\alpha) t)\star\pi(\alpha^*)=\lambda\alpha t;\\\pi(\alpha)\star(\pi(\beta)\star\pi(\alpha))&=\pi(\alpha)\star(\lambda t(\beta) t+\kappa C_\beta^{\m(C_\beta)} t)=(\lambda \alpha t+\kappa \beta^* t^{2}),            \end{align*} where $t(\alpha)$ and $t(\beta)$ denotes the idempotent corresponding to the terminus vertex of $\alpha$ and $\beta$ respectively. Therefore, 
		$$(\pi(\alpha)\star\pi(\beta))\star\pi(\alpha)\neq \pi(\alpha)\star(\pi(\beta)\star\pi(\alpha)).$$
		Then by \cite[Theorems 7.39 and 7.44]{BW}, this reduction system is not a formal deformation of $R_\Gamma$.  
		
		This demonstrates that two type (D) infinitesimal deformations cannot be simultaneously lifted to formal deformations. This phenomenon has a natural interpretation in the surface model (see Subsection \ref{subsection:typeD}): both deformations correspond to deformations of a boundary component with winding number $-2$, but the same boundary cannot be simultaneously compactified into both a smooth point and a singular point.
	\end{remark}
	
	We give the deformations of local Brauer tree algebras in the following example.
	
	\begin{example}\label{exam:localbrauer}
		Let $B=\kk\langle x\rangle/\langle x^{m+1}\rangle$, which is a local Brauer tree algebra with an exceptional vertex $v$ such that $\m(v)=m$. Then each $2$-cochain is of the following general form
		$$\tphi_{x^{m+1}}=\lambda+\sum_{i=1}^{m-1}\varepsilon_i x^i+\nu x^m.$$
		It is obvious that each group of coefficients $(\lambda,\varepsilon_1,\cdots,\varepsilon_{m-1},\nu)$ induces a $2$-cocycle of $B$. 
		
		Moreover, each $1$-cochain has a general form as
		$$\psi(x)=a+\sum_{i=1}^{m}b_i x^i.$$
		Then, for two different $2$-cocycles $\tphi$ and $\tphi'$ with $$\tphi_{x^{m+1}}=\nu x^m,$$$$\tphi'_{x^{m+1}}=\nu' x^m,$$ their difference is a coboundary:  $\tphi-\tphi'=\langle\partial^1(\psi_0)\rangle$ (see Section \ref{sec:differential} for the explicit formula of $\partial^1$). This implies that the coefficient $\nu$ can be adjusted arbitrarily by adding a coboundary. Specifically, the $1$-cochain $\psi_0$ defined by $\psi_0(x)=\upsilon_0$ (for some given $\upsilon_0\in k^*$) yields a $2$-coboundary $\partial^1(\psi_0)$ that changes the value of $\nu$. Therefore, $\nu$ corresponds to a trivial class in the second Hochschild cohomology group, meaning it represents a $2$-coboundary. Thus $\dim_\kk\hh^2(B)=m$.
		Moreover, the deformations of $B$ can be represented by 
		$$B=\kk\langle x\rangle/\langle x^{m+1}+\lambda+\sum_{i=1}^{m-1}\varepsilon_i x^i\rangle.$$ 
	\end{example}
	
	\begin{remark}
		If $\Gamma$ is a Brauer tree with an exceptional vertex $v\in V(\Gamma)$, then it follows from Corollary \ref{dimhh2} and Example \ref{exam:localbrauer} that  $\dim_\kk\hh^2(B_\Gamma)=\m(v)$, which was already shown in \cite[Theorem 4.4]{Holm}.
	\end{remark}
	
	By regarding each ribbon graph as an admissible arc system on its associated surface model, the following example shows that each punctured surface with at least two punctures admits an admissible arc system which can be regarded as a bipartite ribbon graph. 
	Moreover, the dimension formula in Corollary \ref{dimhh2} can be expressed in terms of the surface model's geometric data, such as its boundaries and genus.
	
	\begin{example}\label{exa:bound of graded}
		Let $\Sigma$ be a punctured surface with at least two punctures. Let $\Gamma$ be any admissible arc system of $\Sigma$ and $B_\Gamma$ the associated multiplicity-free Brauer graph algebra. Consider the standard admissible arc system $\Delta$ as illustrated in Figure \ref{fig:bBGA-formal-form}. Note that $\Delta$ is bipartite. 
		\begin{figure}[ht]
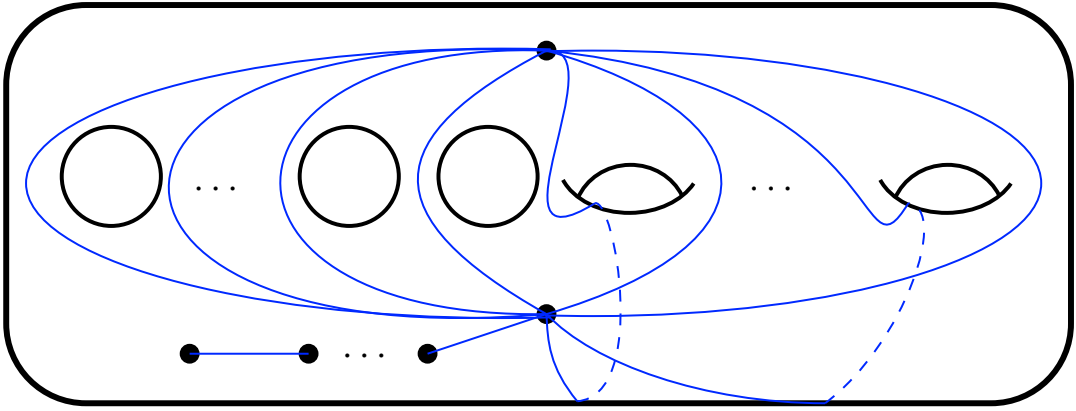

			\centering
			
			\tikzset{every picture/.style={line width=0.75pt}} 
			

			\caption{Bipartite ribbon graph for each punctured surface with at least two punctures.}
			\label{fig:bBGA-formal-form}
		\end{figure}
		Then it follows from \cite[Corollary 6.5]{OZ} that $B_\Gamma$ is derived equivalent to the graded algebra $B_\Delta$ where the grading of $B_\Delta$ is induced by the line field associated to $B_\Gamma$.   
		Denote by $b$ (resp.\ $g$) the number of boundaries (resp.\ genus) of $\Sigma$ (where punctures are not considered as boundaries of $\Sigma$). Let $\partial \Sigma_{-2}$ be the set of boundaries whose winding numbers are equal to $-2$. Then since the deformations of type (A) and type (C) are induced by homogeneous elements, and the deformations of type (D) are totally depended on the winding numbers of boundaries, it follows that these cocyles are still well-defined and lie in $\hh^2(B_\Delta)$ even though $B_\Delta$ is graded. Therefore, we have
		$$\dim_\kk\hh^2(B_\Gamma)= \dim_\kk\hh^2(B_\Delta)\geq 2g+b+2|\partial\Sigma_{-2}|.$$
		Note that $2g+b-1 = |E|-|V|+1= \mathrm{rank} (\mathrm H^1(\Sigma))$, compare to the formula \eqref{align:formalhh}. 
		
		In general, the above inequality may be strict if $\Gamma$ is either non-partite and concentrated in degree zero, or partite with nontrivial grading, see Example \ref{example:annuli-with-2-punctures}.
		
	\end{example}

	\section{Deformations of surface models of bipartite Brauer graph algebras}\label{sec:4}
	
	According to \cite{Kel}, studying the deformation of an associative algebra can be viewed as studying the deformation of its (dg) derived category. Meanwhile, based on \cite{OZ}, the derived category of a Brauer graph algebra can be regarded as the homotopy category of the twisted complex category of the Brauer graph $A_\infty$-category corresponding to its surface model. Therefore, we can attempt to understand the deformation of each (bipartite) Brauer graph algebra through the deformation of its surface model. 
	
	In this section, we give geometric explanations of the deformations of bipartite Brauer graph algebras obtained in Theorem \ref{thm:infinitesimaldeformation}. We assume that $B_\Gamma$ is not a local Brauer tree algebra.
	
	By \cite{GSS} and \cite[Section 5]{OZ}, each Brauer graph $(\Gamma, \m)$ (with $\m\neq 1$) admits a branched cover $\Gamma^\m$, which is a multiplicity-free ($\m=1$) Brauer graph and the original Brauer graph category $B_\Gamma$ is an orbit category of the Brauer graph $A_\infty$-category $B_{\Gamma^\m}$ associated to the branched cover.  It follows from Theorem \ref{thm:infinitesimaldeformation} that the deformations of type (B) only depend on the multiplicity function $\m$. From this point of view, the deformations of type (B) for $B_\Gamma$ shall be induced from the deformations of type (A), (C) and (D) of $B_{\Gamma^\m}$. For this, we need to compare the relationship between the Hochschild cohomology of an $A_\infty$-category and its orbit category,  which will be explored in future work.

	
	Therefore, in this section, we will only focus on deformations of type (A), (C) and (D) and give geometric explanations for these three types.

	\subsection{Deformations of type (A)}\label{subsection:typeA}
	
	Let $(\Gamma=\Gamma[V_1,V_2],\m)$ be a bipartite Brauer graph. For each edge $\bar h\in E(\Gamma)$, $\bar h$ connects a vertex $v\in V_1$ to some $w\in V_2$. Denote by $C_{v,\bar h}$ the cycle of the form $C_\alpha$ where $\alpha$ is the arrow starting from $\bar h$ around $v$, and  the cycle of the form $C_\beta$ where $\beta$ is the arrow starting from $\bar h$ around $w$ by $C_{w,\bar h}$. The generic deformation of type (A) is actually given by replacing each relation $C_{v,\bar h}^{\m(v)}-C_{w,\bar h}^{\m(w)}=0$ by the relation $C_{v,\bar h}^{\m(v)}-C_{w,\bar h}^{\m(w)}+e_{\bar h}=0$. This is a generalized form of the deformation of Brauer tree algebras in \cite[Section 9.A]{BW}.
	
	Denote the algebra induced by the generic deformation of type (A) by $B_\Gamma^{(A)}$. In general, we can deform each modified Brauer graph algebra $B_\Gamma$ which is induced by a bipartite Brauer graph to get a $B_\Gamma^{(A)}$ in the same way.

	The following proposition shows that the deformation of type (A) will make each bipartite Brauer graph algebra become semisimple.
	
	\begin{proposition}\label{semi A}
		Let $B_\Gamma$ be a bipartite Brauer graph algebra which is not a local Brauer tree algebra. Then 
		$$B_\Gamma^{(A)}\cong \prod_{v\in V(\Gamma)}\mathrm{M}_{\mathrm{val}(v)}(\kk)^{\m(v)}.$$
		Therefore, $B_\Gamma^{(A)}$ is Morita equivalent to $\kk^N$ with $N=\sum_{v\in V(\Gamma)}\m(v)$.
	\end{proposition}
	
	\begin{proof}
		For each arrow $\alpha\in Q_1$, define $$\theta(\alpha)= \left\{
		\begin{array}{*{3}{lll}}
			1, & \text{if $\alpha$ is an arrow around some vertex in $V_1$;}\\
			0, & \text{otherwise.}
		\end{array}
		\right.$$ For each $v\in V(\Gamma)$, let $\alpha$ be an arrow around $v$, we also denote $\theta(\alpha)$ by $\theta(v)$.
		We note the following formula: $$((-1)^{\theta(\alpha)}C_\alpha^{\m(C_\alpha)})^2=C_\alpha^{\m(C_\alpha)}(C_\beta^{\m(C_\beta)}+(-1)^{\theta(\alpha)}e)=(-1)^{\theta(\alpha)}C_\alpha^{\m(C_\alpha)},$$  since the relations $C_\alpha^{\m(C_\alpha)}-C_\beta^{\m(C_\beta)}-(-1)^{\theta(\alpha)}e=0$ and $\alpha\beta=0$ hold in $B_\Gamma^{(A)}$ with $e:=t(C_\alpha)$, the idempotent corresponding to the terminus vertex of $C_\alpha$. Therefore, such a $(-1)^{\theta(\alpha)}C_\alpha^{\m(C_\alpha)}$ is an idempotent.

		Consider the set of primitive pairwise orthogonal idempotents $\mathcal{S}=\{(-1)^{\theta(\alpha)}C_\alpha^{\m(C_\alpha)}\;|\;\alpha\in Q_1\}$. Then $\sum_{e\in \mathcal{S}}e=1$.
		
		For each $v\in V(\Gamma)$ with $\mathrm{val}(v)=n$, let $Q_v$ be the following $n$-cycle and $I_v$ be the ideal of $\kk Q_v$ generated by $\{(\alpha_{i-1}\alpha_{i-2}\cdots\alpha_{i-n})^{\m(v)}-(-1)^{\theta(v)}e_i\}_{i\in \mathbb{Z}_n}$.
		$$
		\begin{tikzcd}
			1 \arrow[r, "\alpha_1"] & 2 \arrow[d, no head, dotted]  \\
			0 \arrow[u, "\alpha_0"] & n-1 \arrow[l, "\alpha_{n-1}"]
		\end{tikzcd}$$	
		In fact, the vertices in $Q_v$ correspond to the idempotents $(-1)^{\theta(\alpha)}C_\alpha^{\m(C_\alpha)}$ where $C_\alpha$ are cycles around $v$, and arrows in $Q_v$ correspond to the arrows around $v$ in $Q_\Gamma$. It can be verified that $B_\Gamma^{(A)} \cong \prod_{v \in V(\Gamma)} B_v$ as $\kk$-algebras, where $B_v = \kk Q_v / I_v$. More explicitly, this isomorphism is induced by natural bijections between the idempotents in $\mathcal{S}$ and the vertices $\bigsqcup_{v \in V(\Gamma)} (Q_v)_0$, as well as between the arrows of $Q_\Gamma$ and $\bigsqcup_{v \in V(\Gamma)} (Q_v)_1.$
		
		Since $\kk$ is algebraically closed, denote the solution set of the equation $x^{\m(v)}-(-1)^{\theta(v)}=0$ by $\{\varpi_1,\cdots,\varpi_{\m(v)}\}$ and $I_{v,j}=\langle\alpha_{i-1}\alpha_{i-2}\cdots\alpha_{i-n}-\varpi_je_i\rangle_{i\in \mathbb{Z}_n}$, $1\leq j\leq \m(v)$. Then for each $j$, $\kk Q_v/I_{v,j}\cong \mathrm{M}_n(\kk)$. Actually, $I_v=\bigcap_{j=1}^{\m(v)}I_{v,j}$. By the Chinese Remainder Theorem (c.f. \cite[Corollary 2.27]{Hung}),
		$$B_\Gamma^{(A)}\cong \prod_{v\in V(\Gamma)} B_v\cong\prod_{v\in V(\Gamma)}(\prod_{j=1}^{\m(v)}\kk Q_v/I_{v,j})\cong\prod_{v\in V(\Gamma)}\mathrm{M}_{\mathrm{val}(v)}(\kk)^{\m(v)}.$$
		By Proposition \ref{dim}, it is obvious that $\dim_\kk(B_\Gamma)=\dim_\kk(B_\Gamma^{(A)})$.
	\end{proof}
	
	\begin{corollary}\label{pre der}
		Let $\Gamma$ and $\Gamma'$ be bipartite Brauer graphs (with at least two edges by convention). If $B_\Gamma$ and $B_{\Gamma'}$ are derived equivalent, then $B_\Gamma^{(A)}$ and $B_{\Gamma'}^{(A)}$ are Morita equivalent.
	\end{corollary}
	
	\begin{proof}
		By \cite[Theorem 1.2]{AZ}, if $B_\Gamma$ and $B_{\Gamma'}$ are derived equivalent, then the number of vertices and the multi-sets of multiplicities of $\Gamma$ and $\Gamma'$ coincide. By Proposition \ref{semi A},  $B_\Gamma^{(A)}$ and $B_{\Gamma'}^{(A)}$ are Morita equivalent.
	\end{proof}
	
	\medskip
	\noindent\textbf{Conclusion 1.}\;\; Deformations of type (A) make bipartite Brauer graph algebras become semisimple.
	
	\medskip
	
	\begin{remark}
		Let $(\Sigma, \mathscr{P}, \eta_\Gamma)$ be a graded punctured surface with a line field $\eta_\Gamma$ of ribbon type induced by a bipartite ribbon graph $\Gamma$. For any admissible arc system $\mathcal{A}$ on $(\Sigma, \mathscr{P}, \eta_\Gamma)$, the associated Brauer graph $A_\infty$-category $\mathbb{B}(\mathcal{A}, \m)$ is derived equivalent to the (ungraded) Brauer graph algebra $B_\Gamma$. 
		Therefore, a type (A) deformation of $B_\Gamma$ corresponds to a deformation of $\mathbb{B}(\mathcal{A}, \m)$ that only deforms the commutative relations by adding a multiple of the idempotent, in a manner analogous to type (A) deformations. By Proposition \ref{semi A}, this deformation renders the $A_\infty$-category $\mathbb{B}(\mathcal{A}, \m)$ semisimple up to derived equivalence.
	\end{remark}
	
	Note that the line field associated with a bipartite $\Gamma$ is always orientable, as shown in \cite[Lemma 7.9]{OZ}. We conjecture that the true factor influencing the existence of a semisimple deformation is the orientability of the line field (that is, whether the line field is induced by a vector field), no matter whether $\Gamma$ is bipartite. 
	See the following example. We also refer to Examples \ref{example:annulus} and \ref{example:torus} where the line field is not orientable and the algebra does not admit semisimple deformations. 
	
	\begin{example}\label{exa:graded-simple}
		Consider the Brauer graph ($\Gamma$,1), where $\Gamma$ is given by the following arc system. Assume that the winding numbers at the two boundaries are $-2$ and $0$, respectively.
		
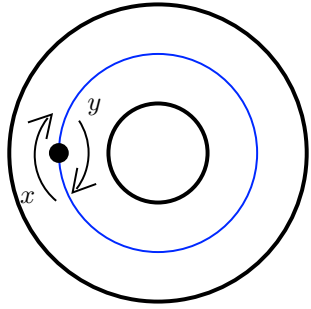
\begin{center}

			\tikzset{every picture/.style={line width=0.75pt}} 
			

			
		\end{center}
		The corresponding modified (graded) Brauer graph algebra is  $$B_\Gamma=\kk\langle x,y\rangle/\langle xy+yx,x^2,y^2\rangle$$ with $|x|=-1$ and $|y|=1$. Here, we remind that the sign $+$ in the relation $xy + yx $ follows from  $-(-1)^{\omega(\bar h)} = 1$, see Definition \ref{defintion:modified}.  Note that $B_\Gamma$ admits a deformation 
		$$B_\Gamma'=\kk\langle x,y\rangle/\langle xy+yx-1,x^2,y^2\rangle,$$
		and there is an isomorphism of algebras   $B_\Gamma'\cong \mathrm{M}_2(\Bbbk)$ given by 
		\begin{align*}
			x \mapsto e_{12}=\left(%
\right).
		\end{align*}  Here, $\mathrm{M}_2(\Bbbk)$ is viewed as a graded matrix algebra with $|e_{12}|=-1$ and $|e_{21}|=1.$ Note that $\Gamma$ is not bipartite and the line field is orientable. 
	\end{example}
	
	\subsection{Deformations of type (C)}\label{subsec:typeC-4.2}
	
	From Theorem \ref{thm:infinitesimaldeformation} it follows that the deformed algebras of type (C) are actually the quantized Brauer graph algebras introduced in \cite{GSS}. 
	
	Fix a spanning tree $T$ of $\Gamma$. For each edge $\bar h \in E(\Gamma)$ which is not in $T$ connecting a vertex $v\in V_1$ to  $w\in V_2$, we assume that $C_{v,\bar h}$ (resp.\ $C_{w,\bar h}$) is the cycle starting from $\bar h$ around $v$ (resp.\ $w$). The generic deformation of type (C) can be chosen by deforming the relation $C_{v,\bar h}^{\m(v)}-C_{w,\bar h}^{\m(w)}=0$ in $I_\Gamma$ into the new relation $C_{v,\bar h}^{\m(v)}+C_{w,\bar h}^{\m(w)} =0$. This deformation can be obtained by deforming the line field $\eta$.  More precisely, consider the cohomology of the ribbon graph $\Gamma$, we have $\dim_\kk \mathrm{H}^1(\Gamma)=|E|-|V|+1$ and the edges which are not in $T$ corresponding to the generators of $\mathrm{H}^1(\Gamma)$ (see for example \cite[Lemma 8.2]{Dew}). Since the ribbon graph $\Gamma$ is a deformation retract of its associated ribbon surface $\Sigma$, these edges will correspond to the non-trivial generators in $\mathrm{H}^1(\Sigma)$, the first cohomology group of $\Sigma$. By \cite[Lemma 1.2]{LP}, these generators can act on the homotopy class of line fields on $\Sigma$ through the short exact sequence 
	\[
	0 \to \mathrm H^1(\Sigma) \to \mathrm H^1(\mathbb P (T\Sigma)) \to \mathrm H^1(S^1) \to 0.
	\]
	This can be understood as deformations of the line fields. Denote by $\eta_{\bar h}$ the line field after the action given by the generator corresponding to the edge $\bar h$ which is not in $T$. Then we have  $$\omega_{\eta_{\bar h}} (\bar h) = \omega_\eta(\bar h) +1.$$  In partiuclar, we obtain that $\omega_{\eta_{\bar h}} (\bar h)$ becomes odd (note that under the original line field $\eta$ the winding number $\omega(\bar h) $ of $\bar h$ is even).  In this way, we obtain the modified relation $C_{v,\bar h}^{\m(v)}-(-1)^{\omega_{\eta_{\bar h}} (\bar h)} C_{w,\bar h}^{\m(w)} = 0$.
	
	
	Indeed, if two punctures $p_1,p_2\in\mathscr{P}$ are connected with more than one arc, then there must be some boundaries or genera between them. For example, each boundary can provide $1$-dimensional deformation of line field and each genus can provide $2$-dimensional deformations of line field which are corresponding to generators in $\mathrm{H}^1(\Sigma)$ (see Figure \ref{fig:line-field-def}).

	\begin{figure}[H]
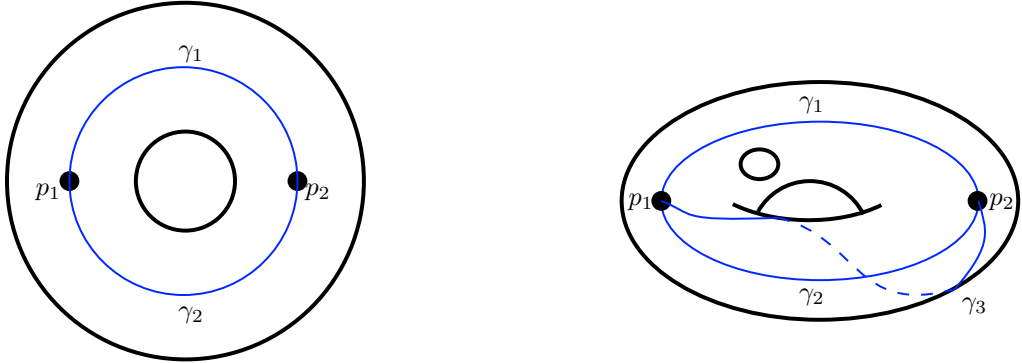

		\centering

		\tikzset{every picture/.style={line width=0.75pt}} 
		

		\caption{ Boundaries and genera can provide generators in $\mathrm{H}^1(\Sigma)$. }
		\label{fig:line-field-def}
	\end{figure}

	\medskip
	
	\noindent\textbf{Conclusion 2.}\;\; Deformations of type (C) are induced from deformations of line fields.
	
	\subsection{Deformations of type (D)}\label{subsection:typeD}
	
	Let $(\Gamma=\Gamma[V_1,V_2],\m=1)$ be a bipartite Brauer graph. Actually, deformations of type (D) are given by some bigon which contains a boundary (since we only consider the Brauer graph algebra with degree zero, the winding number of this boundary is equal to $-2$) in its interior.
	\begin{figure}[H]
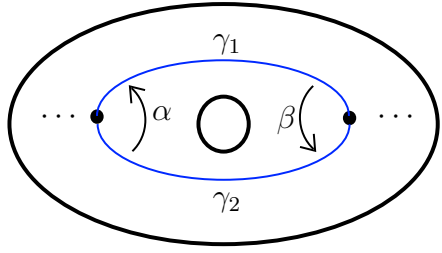

		\centering
		\tikzset{every picture/.style={line width=0.75pt}} 
		

		\caption{The boundary in the bigon which can induce deformations of type (D).}
		\label{fig:(D)-in-geomrtry}
	\end{figure}
	\noindent Since $\alpha\beta$ and $\beta\alpha$ may not be elements in degree $0$ when the bipartite Brauer graph algebra is graded, if we want to give some specific description of the deformations on Brauer graph categories, we still need to give a specific description of the higher Hochschild cohomology group of bipartite Brauer graph algebras in the future. However, through observation and experience, we will show how these deformations behave differently on surface models of Brauer graph algebras.
	
	\subsubsection{Deformations induced by $\lambda'$}
	
	The generic deformation of type (D), induced by the 2-cocycle corresponding to the coefficient $\lambda'$ (that is, setting $\kappa'=0$ in Theorem \ref{thm:infinitesimaldeformation} (D)), is given by replacing  $\alpha\beta$, $\beta\alpha$ in $I_\Gamma$ by $\alpha\beta-e_1$, $\beta\alpha-e_2$, where $e_i$ is the idempotent corresponding to the arc $\gamma_i$. Actually, $\alpha$, $\beta$ can be regarded as a disc sequence in the surface removing this boundary from $\Sigma_\Gamma$. The relations $\alpha\beta=e_1$, $\beta\alpha=e_2$ induce the following multiplications
	\begin{align*}
	\mu(\alpha,\beta)& =e_1;\\
	\mu(\beta,\alpha) & =e_2;\\
	\mu(p,q) &= \begin{cases}
	\beta^*,& \text{ for $pq=C_\alpha\alpha$;}\\
	\alpha^*, & \text{ for $pq=C_\beta\beta$.}
	\end{cases}
	\end{align*}
	Therefore, by \cite{HKK} and \cite[Proposition 6.4]{OZ}, this deformation can be regarded as the following operations on the surface model.
	\begin{figure}[H]
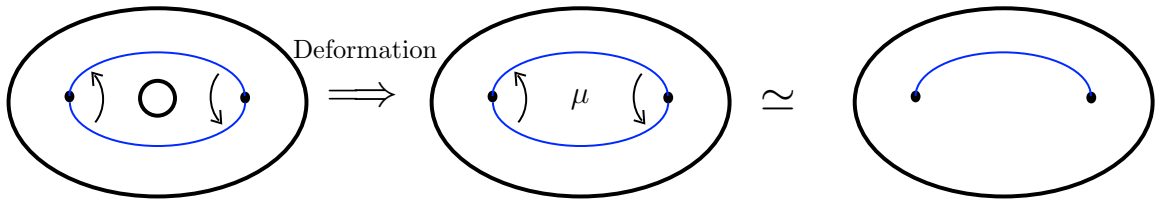

		\centering
		\tikzset{every picture/.style={line width=0.75pt}} 
		

		\caption{Deformations induced by $\lambda'$ in surface model.}
		\label{fig:lamda-in-2}
	\end{figure}

	\medskip
	In fact, these deformations can be generalized in higher multiplications. Let $a_n,\cdots,a_1$ be a sequence of morphisms given by a polygon with exactly one boundary whose winding number is equal to $-2$ in it. Then  there exists a deformation given by $$\mu(ba_n,\cdots,a_1)=b, \;\text{for $ba_n\neq0$};$$
	$$\mu(a_n,\cdots,a_1b)=(-1)^{|b|}b, \;\text{for $a_1b\neq0$};$$
	$$\mu(a_n,\cdots,a_{r+1},a_r(ba_r)^*,ba_r,a_{r-1},\cdots, a_2)=(-1)^\circ a_1^*,\;\text{for $ba_r\neq 0$,}$$
	where $\circ=|a_1|+|a_2|\cdots+|a_{r-1}|+|ba_r|+\omega_\eta(\delta_2)+\cdots+\omega_\eta(\delta_r)$.
	Therefore, by \cite{HKK} and \cite[Proposition 6.4]{OZ}, this deformation can be regarded as the following operations on the surface model.
	\begin{figure}[H]
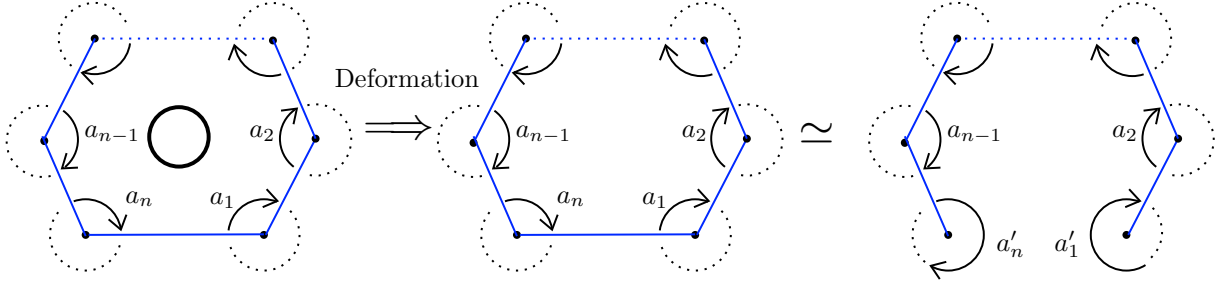

		\centering
		\tikzset{every picture/.style={line width=0.75pt}} 
		

		\caption{The smooth compactification corresponding to the $A_\infty$-deformations induced by $\lambda'$.}
		\label{fig:lambda-in-higher}
	\end{figure}

	\noindent\textbf{Conclusion 3.}\;\; Deformations of type (D) induced by $\lambda'$ correspond to compactifying the  boundary component of winding number $-2$ into a smooth point.
	
	\subsubsection{Deformations induced by $\kappa'$}

	Without loss of generality, we assume $\alpha$ is an arrow around some vertex in $V_1$. The generic deformation of type (D), induced by the 2-cocycle corresponding to the coefficient $\kappa'$ (that is, setting $\lambda'=0$ in Theorem \ref{thm:infinitesimaldeformation}), is given by replacing  $\beta\alpha$ in $I_\Gamma$ by $\beta\alpha-C_\beta$. Actually, $\alpha$, $\beta$ can be regarded as a disc sequence in the surface removing this boundary from $\Sigma_\Gamma$. The relations $\beta\alpha=C_\beta$, $\alpha\beta=0$ induce the following multiplications
	$$\overset{\times}{\mu}(\alpha,\beta)=0;\;\overset{\times}{\mu}(\beta,\alpha)=C_\beta.$$
	Inspired by \cite{BSW}, after deforming this algebra, we may regard this boundary as a {\it special} point (with singularity). This deformation can be regarded as the following operations on the surface model.
	\begin{figure}[H]
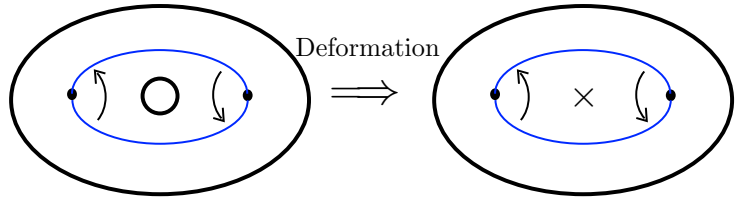

		\centering
		\tikzset{every picture/.style={line width=0.75pt}} 
		

		\caption{Deformations induced by $\kappa'$ in surface model.}
		\label{fig:kappa-in-2}
	\end{figure}
	
	\noindent\textbf{Conclusion 4.}\;\; Deformations of type (D) induced by $\kappa'$ are deformations given by compactifying the boundary component into a `singular' point.
	
	\medskip
	
	In fact, these deformations can be generalized in some higher multiplications. Let $a_n,\cdots,a_1$ be a sequence of morphisms given by a polygon with exactly one boundary whose winding numbers are equal to $-2$ in it. If we fix a morphism $a_n$ in each sequence in the polygon as above, then we conjecture that there exists a deformation given by 
	\begin{align*}
		\overset{\times}{\mu}(a_n,\cdots,a_1)&=C_{a_n};\\
		\overset{\times}{\mu}(a_i,\cdots,a_1,a_n,\cdots,a_{i+1})&=0, 
	\end{align*}
	with $i=1,\cdots,n-1$. This may lead to deformations induced by singular points (similar in spirit to, but distinct from, the boundary deformations considered in \cite{BSW}, which focus on the case of winding number $-1$). These deformations can be regarded as the following operations on the surface model.
	\begin{figure}[H]
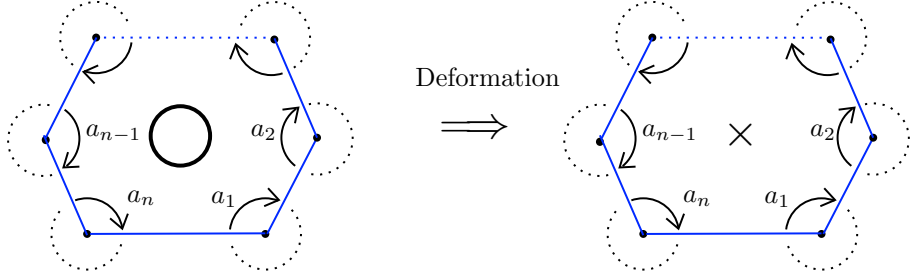

		\centering
		\tikzset{every picture/.style={line width=0.75pt}} 
		

		\caption{Conjecture on deformations induced by $\kappa'$ for each arbitrary morphism sequence.}
		\label{fig:kappa-in-higher}
	\end{figure}

	We should note that for some simple examples, the $A_\infty$-deformation induced by $\kappa'$ corresponds to replacing a puncture with a marked point on the boundary in the surface model.

	\begin{example}\label{exa:BGA-deform-gentle} (Example \ref{exa:graded-simple} revisited)
		Consider the $A_\infty$-deformation of $B_\Gamma$ induced by the derivation $d\in \mathrm{HH}^2(B_\Gamma)$ defined by $d(x)=xy$ (see for example\ in \cite{CSS,LX}). That is, we obtain a dg algebra $(B_\Gamma, d)$, where  $d$ is indeed of degree $1$ since $B_\Gamma$ is graded. Note that this dg algebra is quasi-isomorphic to the graded gentle algebra $k[y]/\langle y^2\rangle$, where $|y|=1$.
	\end{example}

	\section{Examples for non-bipartite Brauer graph algebras}\label{sec:5}
	
	In this section, we give some examples of deformations for non-bipartite Brauer graph algebras. Indeed, for non-bipartite graphs, there is no general method to determine which cycle should be the tip in the commutative relation $C_\alpha^{\m(C_\alpha)} - C_\beta^{\m(C_\beta)} = 0$ (due to the absence of the vertex partition in bipartite case). It is precisely this obstacle—the indeterminacy of the tip in the reduction system—that prevents us from providing a general method to compute. However, by checking some simple examples, we find that there are some different phenomena appeared in this case.
	
	\begin{example}\label{example:annulus}
		Consider the Brauer graph ($\Gamma$,1), where $\Gamma$ is given by the following arc system.
		
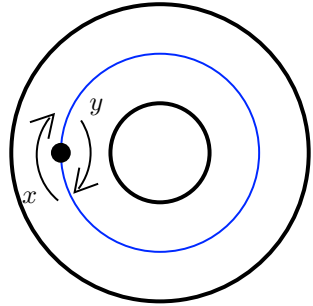
\begin{center}

			\tikzset{every picture/.style={line width=0.75pt}} 
			

			
		\end{center}
		The corresponding Brauer graph algebra is given by $B_\Gamma=\kk\langle x,y\rangle/\langle xy-yx,x^2,y^2\rangle$. Note that $B_\Gamma$ admits a reduction system
		\[
		R_\Gamma = \{(xy, yx), (x^2, 0), (y^2, 0)\}
		\]
		satisfying the condition ($\diamond$). So $S = \{xy, x^2, y^2\}.$
		
		Then the general form of the $2$-cocycles in $\ho(S, B_\Gamma)$ is given by
		\begin{align*}
			\tphi_{xy}& =\kappa yx;\\
			\tphi_{x^2} & =\mu_1+\mu_2 x+\mu_3 y+\mu_4 yx;\\
			\tphi_{y^2} &=\nu_1+\nu_2 x+\nu_3 y+\nu_4 yx.
		\end{align*}
		Similar to the computation in Appendix \ref{appendix:coboundary}, $\tphi$ is a $2$-coboundary if and only if 
		\[
		\kappa=\mu_1=\mu_3=\nu_1=\nu_2 = 0.
		\]
		
		Therefore, $\dim_\kk\hh^2(B_\Gamma)=5$, yielding a family of deformations parameterized by coefficients $\kappa,\mu_1,\mu_3,\nu_1,\nu_2$.

		Actually, the deformation parametrized by $\kappa$ can also be regarded as a deformation of the line field, which is induced by a generator in $\mathrm H^1(\Sigma)$ of the surface model $\Sigma$ (corresponding to the deformation of type (C) in the bipartite case). The deformations parameterized by $\mu_1, \mu_3, \nu_1, \nu_2$ are deformations from the two boundaries in the surface (analogous to the deformations of type (D) in the bipartite case). Note that there are no deformations of type (B) since $\Gamma$ is multiplicity-free. However, deformations of type (A) do not exist either in this example.
	\end{example}

	\begin{example}\label{example:torus}
		Consider the Brauer graph ($\Gamma$,1), where $\Gamma$ is given by the following arc system.
		
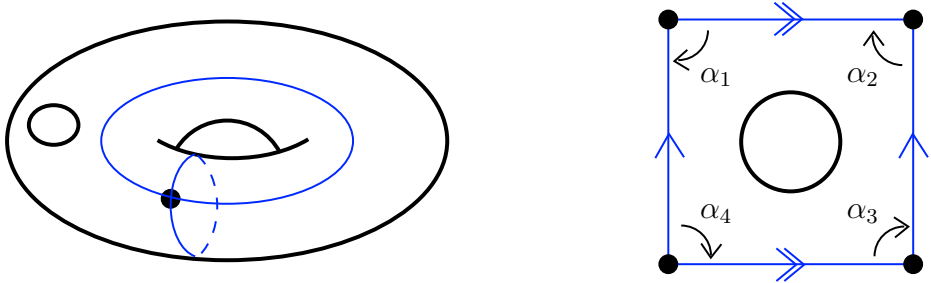
\begin{center}

			\tikzset{every picture/.style={line width=0.75pt}} 
			

			
		\end{center}
		For convenience, we write for example $\alpha_{4321}:=\alpha_{4}\alpha_{3}\alpha_{2}\alpha_{1}$. The corresponding Brauer graph algebra is given by $B_\Gamma=\kk Q/\langle \alpha_{4321}-\alpha_{2143},\alpha_{1432}-\alpha_{3214},\alpha_{41},\alpha_{34},\alpha_{23},\alpha_{12}\rangle$ where the quiver $Q$ is given by
		$$
		\begin{tikzcd}
			1 \arrow[r, "\alpha_1", bend left, shift left=3] \arrow[r, "\alpha_3"', shift right] & 2 \arrow[l, "\alpha_2"', shift right] \arrow[l, "\alpha_4", bend left, shift left=3]
		\end{tikzcd}$$
		We obtain a reduction system of $I$ given by 
		$$R_\Gamma=\{(\alpha_{4321},\alpha_{2143}),(\alpha_{1432},\alpha_{3214}),(\alpha_{32143},0),(\alpha_{41},0),(\alpha_{12},0),(\alpha_{23},0),(\alpha_{34},0)\}.$$
		Note that $\alpha_{32143}=0$ is a redundant relation. 
		By computation, the $2$-cocycles of $B_\Gamma$ are given by
		$$\tphi_{\alpha_{4321}}=\kappa_1 \alpha_{2143};\;\tphi_{\alpha_{1432}}=\kappa_2\alpha_{3214};\;\tphi_{\alpha_{32143}}=0;$$
		$$\tphi_{\alpha_{41}}=\mu_1\alpha_{21}+\mu_2 \alpha_{43}+\kappa_3\alpha_{2143};$$
		$$\tphi_{\alpha_{34}}=\mu_1\alpha_{32}+\mu_3 \alpha_{14}+\kappa_4\alpha_{3214};$$
		$$\tphi_{\alpha_{23}}=\mu_3\alpha_{21}+\mu_4 \alpha_{43}+\kappa_5\alpha_{2143};$$
		$$\tphi_{\alpha_{12}}=\mu_2\alpha_{32}+\mu_4 \alpha_{14}+\kappa_6\alpha_{3214}.$$
		It is a $2$-coboundary if and only if 
		\[
		\kappa_1=\kappa_2=\cdots=\kappa_6= 0.
		\]
		Therefore, $\dim_\kk\hh^2(B_\Gamma)=6$, yielding a family of deformations parameterized by coefficients $\kappa_1,\cdots,\kappa_6$.

		Note that there are no deformations of type (B) since $\Gamma$ is multiplicity-free. In fact, the deformations parametrized by $\kappa_1$ and $\kappa_2$ can also be regarded as deformations of the line fields, which are induced by two generators in $\mathrm H^1(\Sigma)$ of the surface model $\Sigma$ (corresponding to the deformation of type (C) in the bipartite case). The deformations  parametrized by $\kappa_3,\cdots,\kappa_6$ are deformations of the boundary in this surface (corresponding to the deformation of type (D)). However, this boundary provides $4$-dimensional deformations in this example.   And different from the bipartite case, the deformations of type (A) do not exist in this example.
	\end{example}

		\begin{example}\label{example:annuli-with-2-punctures}
		Consider the Brauer graph ($\Gamma$,1), where $\Gamma$ is given by the following arc system in the annulus with two punctures. Note that $\Gamma$ is not bipartite. 
		\medskip
		
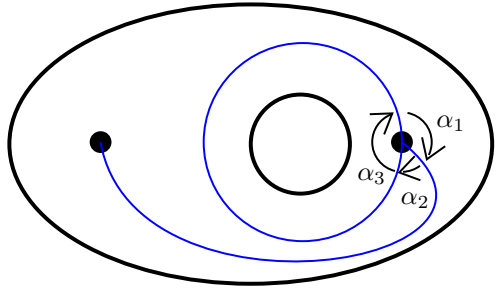
\begin{center}

\tikzset{every picture/.style={line width=0.75pt}} 



		\end{center}
		For convenience, we write for example $\alpha_{321}:=\alpha_{3}\alpha_{2}\alpha_{1}$. The corresponding Brauer graph algebra is given by $B_\Gamma=\kk Q/\langle \alpha_{321}-\alpha_{213},\alpha_{2132},\alpha_{3}^2,\alpha_{12}\rangle$ where the quiver $Q$ is given by
		$$
\begin{tikzcd}
1 \arrow[r, "\alpha_1", shift left] \arrow["\alpha_3"', loop, distance=2em, in=215, out=145] & 2 \arrow[l, "\alpha_2", shift left]
\end{tikzcd}$$
	We choose a line field $\eta$ so that $B_\Gamma$ is concentrated in degree zero. 	We obtain a reduction system of $I$ given by 
		$$R_\Gamma=\{(\alpha_{321},\alpha_{213}),(\alpha_{2132},0),(\alpha_{3}^2,0),(\alpha_{12},0)\}.$$
		Note that $\alpha_{2132}=0$ is a redundant relation. 
		By computation, the $2$-cocycles of $B_\Gamma$ are given by

			\begin{align*}
			\tphi_{\alpha_{321}}& =\kappa \alpha_{213};\\
			\tphi_{\alpha_{2132}} & =0;\\
			\tphi_{\alpha_{3}^2} & =\lambda e_1+\mu \alpha_3+\nu \alpha_{21}+\kappa' \alpha_{213};\\
			\tphi_{\alpha_{12}} &=\xi \alpha_{132}.
		\end{align*}
		It is a $2$-coboundary if and only if 
		\[
		\kappa=\lambda=\nu= 0.
		\]
		Therefore, $\dim_\kk\hh^2(B_\Gamma)=3$, yielding a family of deformations parameterized by coefficients $\kappa,\lambda,\nu$.

		Different from the bipartite case, $B_\Gamma$ does not admit a semisimple deformation (corresponding to type (A) in the bipartite case). Note that there are no deformations of type (B) since $\Gamma$ is multiplicity-free. The family of deformations parametrized by $\kappa$  can be regarded as deformations of the line fields, which are induced by two generators in $\mathrm H^1(\Sigma)$ of the surface model $\Sigma$ (corresponding to  type (C) in the bipartite case).    The deformations  parametrized by $\lambda$ and $\nu$ come from the boundary of winding number $-1$. 
		
		
		Inspired by \cite{BSW1, BSW}, we expect that a certain deformation exists. This deformation can be described by the orbifold surface obtained from the original surface through the collapse of a boundary component with winding number $-1$ into an orbifold point, a direction we plan to explore in future work. Alternatively, this deformation can also be understood as an $A_\infty$-deformation of the following graded (modified) Brauer graph algebra, which provides an intuitive description via the framework of dg algebras.
		
		On the same surface model, we consider the following admissible arc system $\Delta$. 
		\begin{center}

\tikzset{every picture/.style={line width=0.75pt}} 



		\end{center}
		By \cite[Corollary 6.5]{OZ}, the ungraded algebra $B_\Gamma$ is derived equivalent to the graded (modified) Brauer graph algebra $B_\Delta$ corresponding to the arc system $\Delta$. To be more specific, $$B_\Delta=\kk Q/\langle \alpha_{12}+\beta_{12},\alpha_{21}-\beta_{21},\beta_1\alpha_2,\alpha_2\beta_1,\alpha_1\beta_2,\beta_2\alpha_1\rangle$$ where the quiver $Q$ is given by
		$$
		\begin{tikzcd}
			1 \arrow[r, "\alpha_1", bend left, shift left=3] \arrow[r, "\beta_1"', shift right] & 2 \arrow[l, "\alpha_2"', shift right] \arrow[l, "\beta_2", bend left, shift left=3]
		\end{tikzcd}$$
		Note that here $|\alpha_1|=-|\alpha_2|=1$ and $|\beta_1|=|\beta_2|=0$, since the winding numbers at the two boundaries are $-3$ and $-1$.

			Consider the $A_\infty$-deformation induced by the derivation $d \in \mathrm{HH}^2(B_\Delta)$ defined by $d(\alpha_2) = \beta_2$ and $d(\beta_1) = \alpha_1$ (see, for example, in \cite{CSS,LX}). Here, $B_\Delta$ is graded so that $d$ is of degree $1$. Under this deformation, $B_\Delta$ becomes a dg algebra that is quasi-isomorphic to the ungraded algebra $$k[\alpha_{12}]/\langle \alpha_{12}^2\rangle \times k[\alpha_{21}]/\langle \alpha_{21}^2\rangle \cong k[x]/\langle x^2\rangle \times k[x]/\langle x^2\rangle,$$ which is the trivial extension of the semisimple skew-gentle algebra $k[x]/\langle x^2 - x \rangle \cong k \times k$, which can be viewed as a skew Brauer graph algebra in \cite{EGV,Soto}.

            \begin{remark}
	We may choose another line field $\eta'$ so that the Brauer graph algebra, denoted by $B_{\Delta}'$, corresponding to $\Delta$ is concentrated in degree zero. In this case the winding number of the two boundary components both equal $-2$. Then by Corollary \ref{dimhh2} we obtain $\hh^2(B_\Delta')=6$, which is different from the graded case.
		\end{remark}

	\end{example}
	
	\appendix
	
	\section{$2$-Cocycles}\label{appendix:2-cocycle}

By \cite[Section 7.A]{BW} (see also the interpretation in the proof of Theorem \ref{thm:infinitesimaldeformation}), a cochain $\tphi$ is a $2$-cocycle if and only if the following condition holds for every path $uvw \in S_3$ such that $uv, vw \in S$:
	$$(\pi(u)\star\pi(v))\star\pi(w)=\pi(u)\star(\pi(v)\star\pi(w)) \;\; \text{mod $t^2$}.$$
	Here, $\pi: \kk Q \twoheadrightarrow \kk Q/I$ is the natural projection onto the space of irreducible paths $\irr_S$,  and the operation $\star = \star^C_{\varphi+\tphi t}$ denotes multiplication in the deformed algebra $\Bbbk Q[[t]] \big/ \langle s-\varphi_s-\tphi_s t\rangle^{\hat{}}_{s\in S}$. This is performed by multiplying paths in $\kk Q[[t]]$ and then reducing the result to an irreducible path using the reduction system $\widehat{R}_{\varphi+\tphi t}$. This reduction process means that for any path $p$ containing a subpath $s$ (i.e., $p = qsr$), we replace the occurrence of $s$ with $q(\varphi_s + \tphi_st)r$. The reduction process continues iteratively until the result is irreducible.
	
	Let $\gamma=\alpha\sigma(\alpha)\cdots\sigma^i(\alpha)$ be a non-trivial path in $B_\Gamma$. We will denote by $\gamma^*$ the path of $B_\Gamma$ which complements $\gamma$ to a maximal non-zero path, that is $\gamma\gamma^*=C_\alpha^{\m(C_\alpha)}$. Now we discuss all the $1$-ambiguities one by one, by considering the different 
forms of $\tilde{\varphi}_s$ appearing in the general form of the $2$-cochains defined in Section~\ref{sec:differential}, corresponding to type~(1)–(12) elements in $\operatorname{Hom}_{\Bbbk (Q')_0^e}(\Bbbk S, B_\Gamma)$.

	\begin{enumerate}[\textit{Case} 1.]
		\item Let $v\in V_1$, $\m(v)=m$ and $s(\iota(h_1))\neq s(\iota(h_2))$. Then consider the $1$-ambiguity $\alpha\alpha^*\alpha=\alpha C_{\alpha'}^m=C_\alpha^m\alpha$ with $\alpha'=\sigma(\alpha)$, $s_1=C_\alpha^m$, $s_2=C_{\alpha'}^m$. In this case, $\tphi_{s_1}$ and $\tphi_{s_2}$ have forms in (1).
		
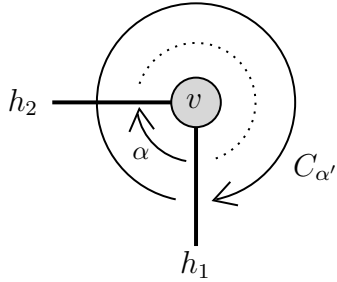
\begin{center}

			\tikzset{every picture/.style={line width=0.75pt}} 
			

			
		\end{center}
		Then $$(\pi(\alpha)\star\pi(\alpha^*))\star\pi(\alpha)=\lambda^{(s_1)}\alpha t+\sum_{i=1}^{m-1}\mu_i^{(s_1)}C_\alpha^i\alpha t+o(t^2);$$
		$$\pi(\alpha)\star(\pi(\alpha^*)\star\pi(\alpha))=\lambda^{(s_2)}\alpha t+\sum_{i=1}^{m-1}\mu_i^{(s_2)}\alpha C_{\alpha'}^i t+o(t^2).$$
		Actually, $C_\alpha^i\alpha=\alpha C_{\alpha'}^i$ in $\kk Q'_\Gamma$. Thus we have
		$$\lambda^{(s_1)}=\lambda^{(s_2)},\;\mu_i^{(s_1)}=\mu_i^{(s_2)}\;(1\leq i\leq m-1).$$ 
		
		\item Let $v\in V_1$, $\m(v)=m$, $s(\iota(h_1))= s(\iota(h_2))$ but $p_\beta\notin (Q'_\Gamma)_1$. Then consider the $1$-ambiguity $\alpha\alpha^*\alpha=\alpha C_{\alpha'}^m=C_\alpha^m\alpha$ with $\alpha'=\sigma(\alpha)$, $s_1=C_\alpha^m$, $s_2=C_{\alpha'}^m$. In this case, $\tphi_{s_1}$ and $\tphi_{s_2}$ have forms in (1).
		
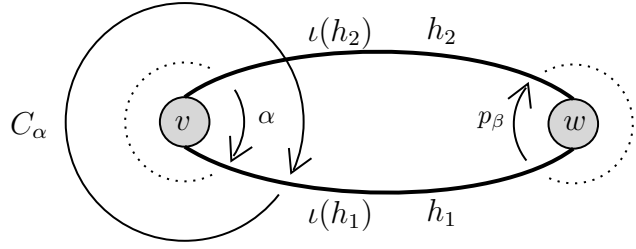
\begin{center}

			\tikzset{every picture/.style={line width=0.75pt}} 
			

			
		\end{center}
		Then $$(\pi(\alpha)\star\pi(\alpha^*))\star\pi(\alpha)=\lambda^{(s_1)}\alpha t+\sum_{i=1}^{m-1}\mu_i^{(s_1)}C_\alpha^i\alpha t+o(t^2);$$
		$$\pi(\alpha)\star(\pi(\alpha^*)\star\pi(\alpha))=\lambda^{(s_2)}\alpha t+\sum_{i=1}^{m-1}\mu_i^{(s_2)}\alpha C_{\alpha'}^i t+o(t^2).$$
		Actually, $C_\alpha^i\alpha=\alpha C_{\alpha'}^i$ in $\kk Q'_\Gamma$. Thus we have
		$$\lambda^{(s_1)}=\lambda^{(s_2)},\;\mu_i^{(s_1)}=\mu_i^{(s_2)}\;(1\leq i\leq m-1).$$
		
		\item Let $v\in V_1$, $\m(v)=m$, and $\beta\in (Q'_\Gamma)_1$. Then consider the $1$-ambiguity $\alpha\alpha^*\alpha=\alpha C_{\alpha'}^m=C_\alpha^m\alpha$ with $\alpha'=\sigma(\alpha)$, $s_1=C_\alpha^m$, $s_2=C_{\alpha'}^m$. In this case, $\tphi_{s_1}$ and $\tphi_{s_2}$ have forms in (1).
		
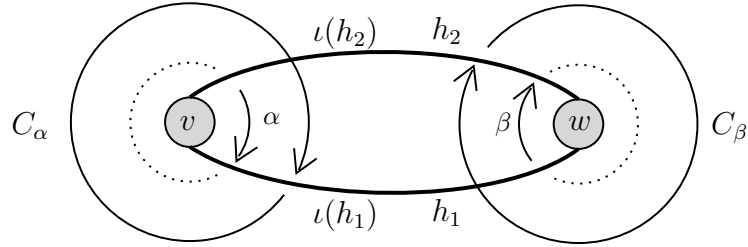
\begin{center}

			\tikzset{every picture/.style={line width=0.75pt}} 
			

			
		\end{center}
		Then $$(\pi(\alpha)\star\pi(\alpha^*))\star\pi(\alpha)=\lambda^{(\beta\alpha)}\beta^*t+\lambda^{(s_1)}\alpha t+\sum_{i=1}^{m-1}\mu_i^{(s_1)}C_\alpha^i\alpha t+o(t^2);$$
		$$\pi(\alpha)\star(\pi(\alpha^*)\star\pi(\alpha))=\lambda^{(\alpha\beta)}\beta^*t+\lambda^{(s_2)}\alpha t+\sum_{i=1}^{m-1}\mu_i^{(s_2)}\alpha C_{\alpha'}^i t+o(t^2).$$
		Actually, $C_\alpha^i\alpha=\alpha C_{\alpha'}^i$ in $\kk Q'_\Gamma$. Thus we have
		$$\lambda^{(\beta\alpha)}=\lambda^{(\alpha\beta)},\;\lambda^{(s_1)}=\lambda^{(s_2)},\;\mu_i^{(s_1)}=\mu_i^{(s_2)}\;(1\leq i\leq m-1).$$
		
		\item Let $w\in V_2$, $\m(w)=n$, $h_1\neq h_2$. Then consider the $1$-ambiguity $\beta'C_\beta^n\beta$ with $\beta'=\sigma^{-1}(\beta)$, $s_1=\beta'C_\beta^n$, $s_2=C_\beta^n\beta$. In this case, $\tphi_{s_1}$, $\tphi_{s_2}$ have forms either in (2) or in (3).
		
		
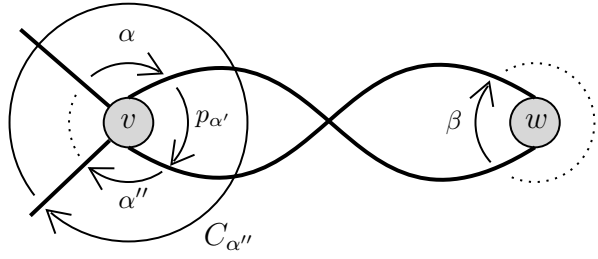
\begin{center}

			\tikzset{every picture/.style={line width=0.75pt}} 
			
	
		\end{center}
		Then $$(\pi(\beta')\star\pi(C_\beta^n))\star\pi(\beta)=\lambda^{(s_1)}\beta'\beta t+\sum_{j=1}^{n-1}\varepsilon_j^{(s_1)}\beta'C_\beta^i\beta t+o(t^2);$$
		$$\pi(\beta')\star(\pi(C_\beta^n)\star\pi(\beta))=\lambda^{(s_2)}\beta'\beta t+\sum_{j=1}^{n-1}\varepsilon_j^{(s_2)}\beta'C_\beta^i\beta t+o(t^2).$$
		Thus we have
		$$\lambda^{(s_1)}=\lambda^{(s_2)},\;\varepsilon_j^{(s_1)}=\varepsilon_j^{(s_2)}\;(1\leq j\leq n-1).$$
		
		\item Let $\m(v)=m$, $\m(w)=n$, $h_1\neq h_2$, $s(\iota(h_1))=s(\iota(h_2))=v$. Denote the cycle starting from $\iota(h_2)$ by $C_{\sigma(\alpha)}$ and the cycle starting from $h_2$ by $C_\beta$. Then consider the $1$-ambiguity $\alpha\beta\alpha$, $s_1=\alpha\beta$, $s_2=\beta\alpha$. In this case, $\tphi_{s_1}$, $\tphi_{s_2}$ have forms both in (6) (or both in (10)).
		\begin{center}

			\tikzset{every picture/.style={line width=0.75pt}} 
			

			
		\end{center}
		Then $$(\pi(\alpha)\star\pi(\beta))\star\pi(\alpha)=\lambda^{(s_1)}\alpha t+\sum_{i=1}^{m-1}\xi_i^{(s_1)}C_\alpha^i\alpha t+o(t^2);$$
		$$\pi(\alpha)\star(\pi(\beta)\star\pi(\alpha))=\lambda^{(s_2)}\alpha t+\sum_{i=1}^{m-1}\xi_i^{(s_2)}C_\alpha^i\alpha t+o(t^2).$$
		Thus we have
		$$\lambda^{(s_1)}=\lambda^{(s_2)},\;\xi_i^{(s_1)}=\xi_i^{(s_2)}\;(1\leq i\leq m-1).$$
		Moreover, if we consider the $1$-ambiguity $\beta\alpha\beta$, we have the following formulas:
		$$\lambda^{(s_1)}=\lambda^{(s_2)},\;\zeta_j^{(s_1)}=\zeta_j^{(s_2)}\;(1\leq j\leq n-1).$$
		
		\item Let $\m(v)=m$, $\m(w)=n$. The $1$-ambiguity $\alpha''\beta\alpha$, $s_1=\alpha''\beta$, $s_2=\beta\alpha$. In this case, $\tphi_{s_1}$, $\tphi_{s_2}$ have forms in (7), (8), (9).
		
		\begin{center}

			\tikzset{every picture/.style={line width=0.75pt}} 
			

			
		\end{center}
		Then $$(\pi(\alpha'')\star\pi(\beta))\star\pi(\alpha)=\sum_{i=0}^{m-1}\xi_i^{(s_1)}C_{\alpha''}^i\alpha''p_{\alpha'}\alpha+o(t^2);$$
		$$\pi(\alpha'')\star(\pi(\beta)\star\pi(\alpha))=\sum_{i=0}^{m-1}\xi_i^{(s_2)}C_{\alpha''}^i\alpha''p_{\alpha'}\alpha+o(t^2).$$
		Thus we have
		$$\xi_i^{(s_1)}=\xi_i^{(s_2)}\;(0\leq i\leq m-1).$$
		
		\item Let $\m(v)=m$, $\m(w)=n$, $\mathrm{val}(w)=1$. The $1$-ambiguity $\alpha'\beta\alpha$, $s_1=\alpha'\beta$, $s_2=\beta\alpha$. In this case, $\tphi_{s_1}$, $\tphi_{s_2}$ have forms in (11) and (12).
		
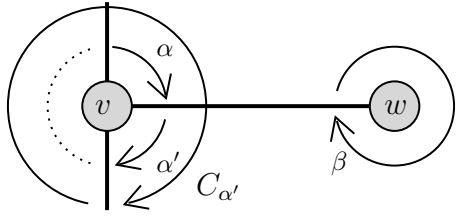
\begin{center}

			\tikzset{every picture/.style={line width=0.75pt}} 
			


		\end{center}
		Then $$(\pi(\alpha')\star\pi(\beta))\star\pi(\alpha)=\sum_{i=0}^{m-1}\xi_i^{(s_1)}C_{\alpha'}^i\alpha'\alpha+o(t^2);$$
		$$\pi(\alpha')\star(\pi(\beta)\star\pi(\alpha))=\sum_{i=0}^{m-1}\xi_i^{(s_2)}C_{\alpha'}^i\alpha'\alpha+o(t^2).$$
		Thus we have
		$$\xi_i^{(s_1)}=\xi_i^{(s_2)}\;(0\leq i\leq m-1).$$
		
		\item  Let $v\in V_1$, $w\in V_2$, $\m(v)=m$, $\m(w)=n$, $s(\iota(h_2))\neq v$. The $1$-ambiguity $\beta\alpha\alpha^*=\beta C_\alpha^m$, $s_1=\beta\alpha$, $s_2=C_\alpha^m$. In this case, $\tphi_{s_1}$, $\tphi_{s_2}$ have forms in (5) and (1).    
		
		\begin{center}
			
			\tikzset{every picture/.style={line width=0.75pt}} 
			

			
		\end{center}
		Then 
		\begin{align*}(\pi(\beta)\star\pi(\alpha))\star\pi(\alpha^*)&=o(t^2);\\\pi(\beta)\star(\pi(\alpha)\star\pi(\alpha^*))&=\lambda^{(C_\beta^n\beta)}\beta t+\sum_{j=1}^{n-1}\varepsilon_j^{(C_\beta^n\beta)}C_\beta^j\beta t+\lambda^{(s_2)}\beta t+\sum_{j=1}^{n-1}\varepsilon_j^{(s_2)}C_\beta^j\beta t+o(t^2).            \end{align*}
		Thus we have
		$$\lambda^{(C_\beta^n\beta)}=-\lambda^{(s_2)},\;\varepsilon_j^{(C_\beta^n\beta)}=-\varepsilon_j^{(s_2)}\;(1\leq j\leq n-1).$$
		
		\item  Let $v\in V_1$, $w\in V_2$, $\m(v)=m$, $\m(w)=n$, $s(\iota(h_2))= v$. The $1$-ambiguity $\beta\alpha\alpha^*=\beta C_\alpha^m$, $s_1=\beta\alpha$, $s_2=C_\alpha^m$. In this case, $\tphi_{s_1}$, $\tphi_{s_2}$ have forms in (6) and (1).    
		
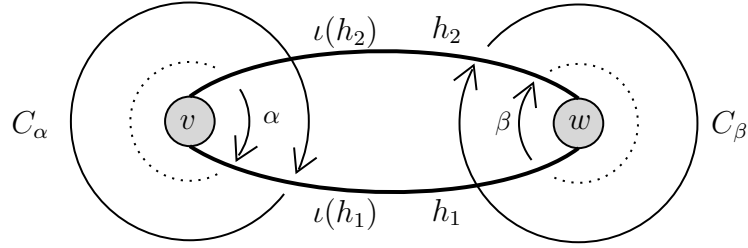
\begin{center}

			\tikzset{every picture/.style={line width=0.75pt}} 
			

			
		\end{center}
		Then 
		\begin{align*}(\pi(\beta)\star\pi(\alpha))\star\pi(\alpha^*)&=\lambda^{(\beta\alpha)}\alpha^*C_\alpha^{m-1} t+o(t^2);\\
			\pi(\beta)\star(\pi(\alpha)\star\pi(\alpha^*))&=\lambda^{(C_\beta^n\beta)}\beta t+\sum_{j=1}^{n-1}\varepsilon_j^{(C_\beta^n\beta)}C_\beta^j\beta t+\sum_{i=0}^{m-1}\nu_i^{(C_\beta^n\beta)}\alpha^*C_\alpha^{i}t\\&\;\;+\lambda^{(s_2)}\beta t+\sum_{j=1}^{n-1}\varepsilon_j^{(s_2)}C_\beta^j\beta t+o(t^2).            \end{align*}
		Thus we have
		$$\lambda^{(C_\beta^n\beta)}=-\lambda^{(s_2)},\;\varepsilon_j^{(C_\beta^n\beta)}=-\varepsilon_j^{(s_2)}\;(1\leq j\leq n-1),$$
		$$\nu_{m-1}^{(C_\beta^n\beta)}=\lambda^{(\beta\alpha)},\;\nu_i^{(C_\beta^n\beta)}=0\;(0\leq i\leq m-2).$$
		
		\item Let $v\in V_1$, $w\in V_2$, $\m(v)=m$, $\m(w)=n$, $s(\iota(h_2))=v$. If $p_{\alpha'}, p_{\beta}\in (Q_\Gamma')_1$, then let $\bar{h_1}\neq \bar{h_2}$. The $1$-ambiguity $\beta'\alpha\alpha^*=\beta' C_\alpha^m$, $s_1=\beta'\alpha$, $s_2=C_\alpha^m$. In this case, $\tphi_{s_1}$, $\tphi_{s_2}$ have forms in (7) and (1).    
		
		
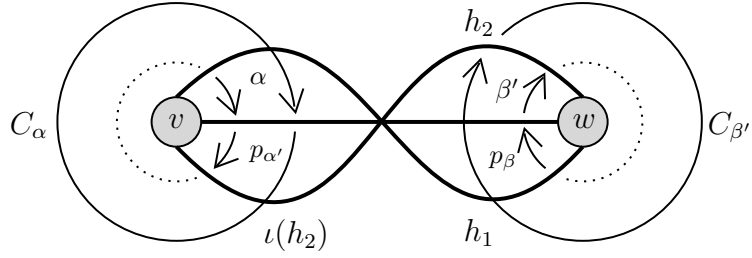
\begin{center}

			\tikzset{every picture/.style={line width=0.75pt}} 
			

			
		\end{center}
		Then 
		\begin{align*}(\pi(\beta')\star\pi(\alpha))\star\pi(\alpha^*)&=o(t^2);\\\pi(\beta')\star(\pi(\alpha)\star\pi(\alpha^*))&=\lambda^{(C_{\beta'}^n\beta')}\beta' t+\sum_{j=1}^{n-1}\varepsilon_j^{(C_{\beta'}^n\beta')}C_{\beta'}^j\beta' t+\sum_{i=0}^{m-1}\nu_i^{(C_{\beta'}^n\beta')}p_{\alpha'}C_\alpha^{i}t\\&\;\;+\lambda^{(s_2)}\beta' t+\sum_{j=1}^{n-1}\varepsilon_j^{(s_2)}C_{\beta'}^j\beta' t+o(t^2).
		\end{align*}
		Thus we have
		$$\lambda^{(C_{\beta'}^n\beta')}=-\lambda^{(s_2)},\;\varepsilon_j^{(C_{\beta'}^n\beta')}=-\varepsilon_j^{(s_2)}\;(1\leq j\leq n-1),\;\nu_i^{(C_{\beta'}^n\beta')}=0\;(0\leq i\leq m-1).$$
		Actually, $\beta'\alpha$ with the form in (8) (or (9)) is same as this case.
		
		\item  Let $v\in V_1$, $w\in V_2$, $\m(v)=m$, $\m(w)=n$, $\mathrm{val}(v)=\mathrm{val}(w)=1$. The $1$-ambiguity $\beta\alpha\alpha^*=\beta\alpha^m$, $s_1=\beta\alpha$, $s_2=\alpha^m$. In this case, $\tphi_{s_1}$, $\tphi_{s_2}$ have forms in (10) and (1).
		\begin{center}
			
			\tikzset{every picture/.style={line width=0.75pt}} 
			

		\end{center}
		Then 
		\begin{align*}
			(\pi(\beta)\star\pi(\alpha))\star\pi(\alpha^*)&=\lambda^{(s_1)}\alpha^{m-1} t+\xi_1^{(s_1)}\beta^n t+o(t^2);\\
			\pi(\beta)\star(\pi(\alpha)\star\pi(\alpha^*))&=\lambda^{(\beta^{n+1})}\beta t+\sum_{j=1}^{n-1}\varepsilon_j^{(\beta^{n+1})}\beta^{j+1} t+\sum_{i=0}^{m-1}\nu_i^{(\beta^{n+1})}\alpha^{i}t\\&\;\;+\lambda^{(s_2)}\beta t+\sum_{j=1}^{n-1}\varepsilon_j^{(s_2)}\beta^{j+1} t+o(t^2). 
		\end{align*}
		Thus we have
		$$\lambda^{(\beta^{n+1})}=-\lambda^{(s_2)},\;\varepsilon_j^{(\beta^{n+1})}=-\varepsilon_j^{(s_2)}\;(1\leq j\leq n-2),$$
		$$\varepsilon_{n-1}^{(\beta^{n+1})}=\xi_1^{(s_1)}-\varepsilon_{n-1}^{(s_2)},\;\nu_{m-1}^{(\beta^{n+1})}=\lambda^{(s_1)},\;\nu_i^{(\beta^{n+1})}=0\;(0\leq i\leq m-2).$$
		
		\item  Let $v\in V_1$, $w\in V_2$, $\m(v)=m$, $\m(w)=n$, $\mathrm{val}(v)\neq 1=\mathrm{val}(w)$. The $1$-ambiguity $\beta\alpha\alpha^*=\beta C_\alpha^m$, $s_1=\beta\alpha$, $s_2=C_\alpha^m$. In this case, $\tphi_{s_1}$, $\tphi_{s_2}$ have forms in (11) and (1).
		
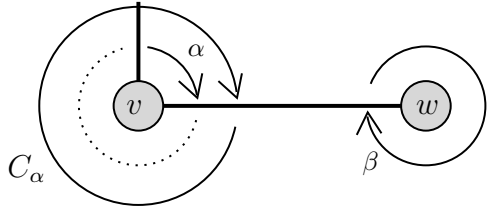
\begin{center}

			\tikzset{every picture/.style={line width=0.75pt}} 
			


		\end{center}
		Then \begin{align*}
			(\pi(\beta)\star\pi(\alpha))\star\pi(\alpha^*)&=\xi_0^{(s_1)}\beta^n t+o(t^2);\\
			\pi(\beta)\star(\pi(\alpha)\star\pi(\alpha^*))&=\lambda^{(\beta^{n+1})}\beta t+\sum_{j=1}^{n-1}\varepsilon_j^{(\beta^{n+1})}\beta^{j+1} t+\sum_{i=0}^{m-1}\nu_i^{(\beta^{n+1})}C_\alpha^{i}t\\&\;\;+\lambda^{(s_2)}\beta t+\sum_{j=1}^{n-1}\varepsilon_j^{(s_2)}\beta^{j+1} t+o(t^2).
		\end{align*}
		Thus we have
		$$\lambda^{(\beta^{n+1})}=-\lambda^{(s_2)},\;\varepsilon_j^{(\beta^{n+1})}=-\varepsilon_j^{(s_2)}\;(1\leq j\leq n-2),$$
		$$\varepsilon_{n-1}^{(\beta^{n+1})}=\xi_0^{(s_1)}-\varepsilon_{n-1}^{(s_2)},\;\nu_i^{(\beta^{n+1})}=0\;(0\leq i\leq m-1).$$
		
		\item  Let $v\in V_1$, $w\in V_2$, $\m(v)=m$, $\m(w)=n$, $\mathrm{val}(v)=1\neq \mathrm{val}(w)$. The $1$-ambiguity $\beta\alpha\alpha^*=\beta\alpha^m$, $s_1=\beta\alpha$, $s_2=\alpha^m$. In this case, $\tphi_{s_1}$, $\tphi_{s_2}$ have forms in (12) and (1).
		
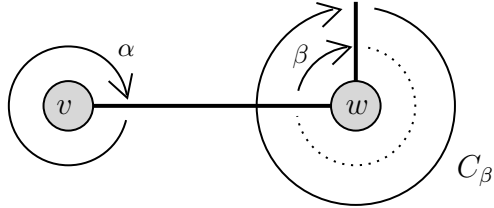
\begin{center}

			\tikzset{every picture/.style={line width=0.75pt}} 
			


		\end{center}
		Then \begin{align*}
			(\pi(\beta)\star\pi(\alpha))\star\pi(\alpha^*)&=o(t^2);\\\pi(\beta)\star(\pi(\alpha)\star\pi(\alpha^*))&=\lambda^{(C_\beta^n\beta)}\beta t+\sum_{j=1}^{n-1}\varepsilon_j^{(C_\beta^n\beta)}C_\beta^{j}\beta t+\lambda^{(s_2)}\beta t+\sum_{j=1}^{n-1}\varepsilon_j^{(s_2)}C_\beta^j\beta t+o(t^2).
		\end{align*}
		Thus we have
		$$\lambda^{(C_\beta^n\beta)}=-\lambda^{(s_2)},\;\varepsilon_j^{(C_\beta^n\beta)}=-\varepsilon_j^{(s_2)}\;(1\leq j\leq n-1),$$
		
	\end{enumerate}
	
	In addition to above cases, the $1$-ambiguities can not provide valuable formulas about the coefficients of cocycles.
	
	\begin{remark}\label{rmk:cocycle}
		We summarize some important cases above as follows.
		\begin{itemize}
			\item Cases 1-4 show that every $\lambda$ (every $\mu$, or every $\varepsilon$) around same vertex in the Brauer graph $\Gamma$ are actually equal to each other.
			
			\item Cases 8-13 show that every coefficients in the relations of type (b) in the reduction system $R_\Gamma$ can be represented by other coefficients in the relations in $R_\Gamma$.
			
			\item The two observations above actually tell us that every $\lambda$ induced by the relations of type (a) and type (b) in $R_\Gamma$ are equal to each other.
		\end{itemize}
	\end{remark}
	
	\section{$2$-Coboundaries} \label{appendix:coboundary}
	
	By \cite[Section 7.A]{BW}, two $2$-cocycles $$\tphi,\tphi'\in \ho_{\kk (Q_\Gamma)_0^e}(\kk S, B_\Gamma)\cong \ho_{\kk (Q_\Gamma)_0^e}(\kk S, \kk\irr_S)$$ are cohomologous, i.e. $\tphi'-\tphi=\langle\psi\rangle$ for some $\psi\in \ho_{\kk (Q_\Gamma)_0^e}(\kk (Q_\Gamma)_1, B_\Gamma)\cong \ho_{\kk (Q_\Gamma)_0^e}(\kk (Q_\Gamma)_1, \kk\irr_S)$, if and only if the isomorphism $T:\kk\irr_S[t]/(t^2)\rightarrow\kk\irr_S[t]/(t^2)$ determined by $T(x)=x+\psi(x)t$ for $x\in (Q_\Gamma)_1$, satisfies
	$$T(\varphi(s))+\tphi'(s)t=T(\alpha_m)\star\cdots\star T(\alpha_1)\;\;\text{mod $t^2$}$$
	for any $s\in S$ with $s=\alpha_m\cdots\alpha_1$ for $\alpha_i\in (Q_\Gamma)_1$.
	
	Actually, the elements $\psi\in\ho_{\kk (Q_\Gamma)_0^e}(\kk (Q_\Gamma)_1, \kk\irr_S)$ have the following forms.
	
	\begin{itemize}
		\item $v\in V(\Gamma)$, $\m(v)=m$, $h_1\neq h_2$, $s(\iota(h_1))\neq s(\iota(h_2))$.
		
		
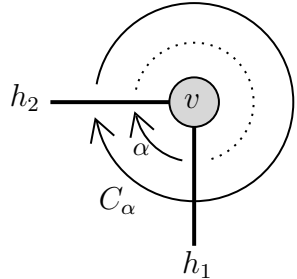
\begin{center}

			\tikzset{every picture/.style={line width=0.75pt}} 
			
	
		\end{center}
		Then $$\psi(\alpha)=\sum_{i=0}^{m-1}a_i^{(\alpha)}C_\alpha^i\alpha.$$
		
		\item Let $v,w\in V(\Gamma)$, $\m(v)=m$, $\m(w)=n$, $h_1\neq h_2$, $s(\iota(h_1))=s(\iota(h_2))=v$. Denote the cycle starting from $\iota(h_2)$ by $C_\beta$.
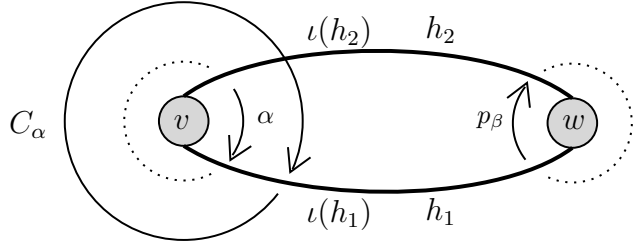
		\begin{center}

			\tikzset{every picture/.style={line width=0.75pt}} 
			
			%
			
		\end{center}
		Then $$\psi(\alpha)=\sum_{i=0}^{m-1}a_i^{(\alpha)}C_\alpha^i\alpha+\sum_{i=0}^{n-1}b_j^{(\alpha)}C_\beta^jp_\beta.$$
		
		\item Let $v,w\in V(\Gamma)$, $\m(w)=n$, $\m(v)=m$, $\mathrm{val}(v)=1$, . Denote the cycle starting from $\iota(h)$ by $C_\beta$.
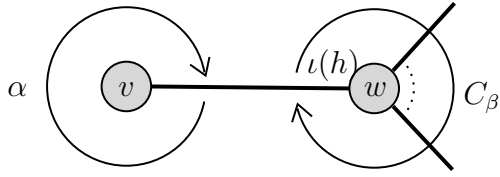
		\begin{center}
			
			\tikzset{every picture/.style={line width=0.75pt}} 
			
			%
		\end{center}
		Then $$\psi(\alpha)=\sum_{i=0}^{m}a_i^{(\alpha)}\alpha^i+\sum_{j=1}^{n-1}b_j^{(\alpha)}C_\beta^j.$$
	\end{itemize}
	
	Now we give some specific discussions about $2$-coboundaries. (We only show the cases where the difference is not zero)
	
	\begin{enumerate}[\textit{Case} 1.]
		\item 	Consider (1) in Section 3.1. Let $v\in V_1$, $\m(v)\mathrm{val}(v)\neq 1$, $w\in V_2$, $\m(v)=m$, $\m(w)=n$, $e$ be the idempotent corresponding to this edge. In this case, $s=C_\alpha^m$, $\varphi_s=C_\beta^n$. Let $C_\alpha=\alpha_k\cdots\alpha_1$, $C_\beta=\beta_l\cdots\beta_1$.
		\begin{center}
			
			\tikzset{every picture/.style={line width=0.75pt}} 
			

		\end{center}
		Then,  $\tphi_s=\lambda^{(s)}e+\sum_{i=1}^{m-1}\mu_i^{(s)}C_\alpha^i+\sum_{j=1}^{n-1}\varepsilon_j^{(s)}C_\beta^j+\kappa^{(s)}C_\beta^n.$ 
		\begin{itemize}
			\item If $\mathrm{val}(v)\neq 1$, $\mathrm{val}(w)\neq 1$, then 
			$$\kappa^{(s)'}-\kappa^{(s)}=m\sum_{i=1}^{k}a_0^{(\alpha_i)}-n\sum_{j=1}^{l}a_0^{(\beta_j)}.$$
			
			\item If $\mathrm{val}(v)=1$, $\mathrm{val}(w)\neq 1$, then 
			$$\mu_{m-1}^{(s)'}-\mu_{m-1}^{(s)}=ma_0^{(\alpha_1)},\;\kappa^{(s)'}-\kappa^{(s)}=ma_1^{(\alpha_1)}-n\sum_{j=1}^{l}a_0^{(\beta_j)}.$$
			
			\item If $\mathrm{val}(v)\neq 1$, $\mathrm{val}(w)= 1$, then 
			$$\varepsilon_{n-1}^{(s)'}-\varepsilon_{n-1}^{(s)}=-na_0^{(\beta_1)},\;\kappa^{(s)'}-\kappa^{(s)}=m\sum_{i=1}^{k}a_0^{(\alpha_i)}-na_1^{(\beta_1)}.$$
			
			\item If $\mathrm{val}(v)=\mathrm{val}(w)= 1$, then 
			$$\mu_{m-1}^{(s)'}-\mu_{m-1}^{(s)}=ma_0^{(\alpha_1)},\;\varepsilon_{n-1}^{(s)'}-\varepsilon_{n-1}^{(s)}=-na_0^{(\beta_1)},\;\kappa^{(s)'}-\kappa^{(s)}=ma_1^{(\alpha_1)}-na_1^{(\beta_1)}.$$
		\end{itemize}

		\item Consider (6) in Section 3.1. Let $\m(v)=m$, $\m(w)=n$, $h_1\neq h_2$, $s(\iota(h_1))=s(\iota(h_2))=v$, $e$ be the idempotent corresponding to $\bar{h_1}$. Denote the cycle starting from $\iota(h_1)$ by $C_{\sigma(\alpha)}$ and the cycle starting from $h_2$ by $C_\beta$. Actually, $C_{\sigma(\alpha)}^m=C_\beta^n$ in $B_\Gamma$. In this case, $s=\beta\alpha$, $\varphi_s=0$.
		
		
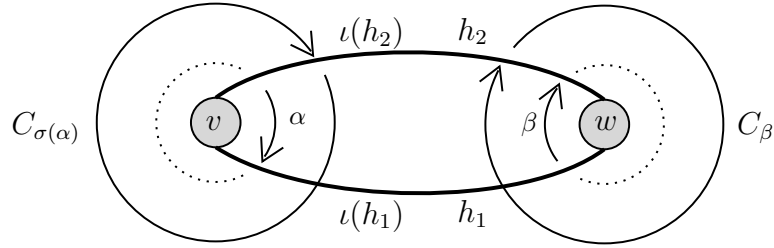
\begin{center}

			\tikzset{every picture/.style={line width=0.75pt}} 
			

			
		\end{center}
		Then $\tphi_s=\lambda^{(s)}e+\sum_{i=1}^{m-1}\xi_i^{(s)}C_{\sigma(\alpha)}^i+\sum_{j=1}^{n-1}\zeta_j^{(s)}C_\beta^j+\kappa^{(s)}C_\beta^n$. In this case,
		$$\xi_i^{(s)'}-\xi_i^{(s)}=b_{i-1}^{(\beta)}\;(1\leq i\leq m-1),\;\zeta_j^{(s)'}-\zeta_j^{(s)}=b_{j-1}^{(\alpha)}\;(1\leq j\leq n-1),\;\kappa^{(s)'}-\kappa^{(s)}=b_{n-1}^{(\alpha)}+b_{m-1}^{(\beta)}.$$
		In fact, the coboundary which is given by $\kappa^{(\alpha\beta)}$ and $\kappa^{(\beta\alpha)}$ induces a deformed algebra which is isomorphic to $B_\Gamma$ (\cite[Lemma 10]{AZ1}).
		
		\item Consider (7) in Section 3.1. Let $\m(v)=m$, $\m(w)=n$. If $p_{\alpha'}, p_{\beta}\in (Q_\Gamma')_1$, then let $\bar{h_1}\neq \bar{h_2}$. Denote the cycle starting from $\iota(h_2)$ by $C_{\alpha'}$ and the cycle starting from $h_2$ by $C_{\beta'}$. In this case, $s=\beta'\alpha$, $\varphi_s=0$.
		
		
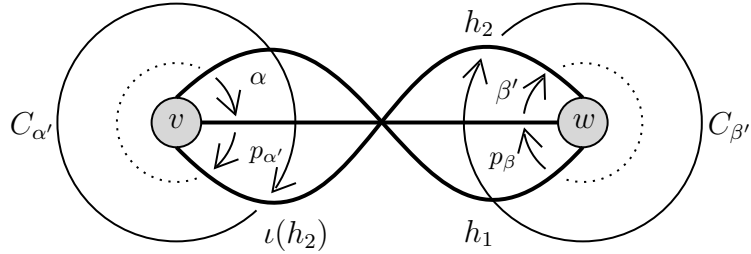
\begin{center}

			\tikzset{every picture/.style={line width=0.75pt}} 
			

			
		\end{center}
		Then $\tphi_s=\sum_{i=0}^{m-1}\xi_i^{(s)}C_{\alpha'}^ip_{\alpha'}\alpha+\sum_{j=0}^{n-1}\zeta_j^{(s)}C_{\beta'}^j\beta'p_{\beta}$. In this case,
		$$\xi_i^{(s)'}-\xi_i^{(s)}=b_{i}^{(\beta)}\;(0\leq i\leq m-1),\;\zeta_j^{(s)'}-\zeta_j^{(s)}=b_{j}^{(\alpha)}\;(0\leq j\leq n-1).$$
		
		\item Consider (8) in Section 3.1. Let $\m(w)=n$, $s(\iota(h_1))=v\neq s(\iota(h_2))$. Denote the cycle starting from $h_2$ by $C_{\beta'}$. In this case, $s=\beta'\alpha$, $\varphi_s=0$.
		
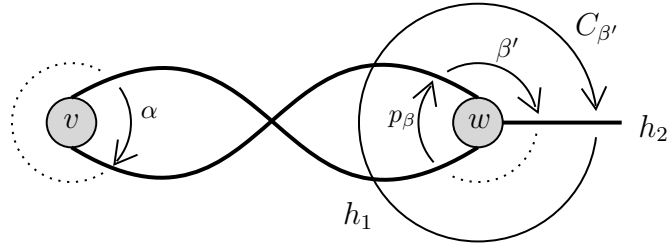
\begin{center}

			\tikzset{every picture/.style={line width=0.75pt}} 
			

			
		\end{center}
		Then $\tphi_s=\sum_{j=0}^{n-1}\zeta_j^{(s)}C_{\beta'}^j\beta'p_\beta$. In this case,
		$$\zeta_j^{(s)'}-\zeta_j^{(s)}=b_{j}^{(\alpha)}\;(0\leq j\leq n-1).$$
		
		\item Consider (9) in Section 3.1. Let $\m(v)=m$, $s(\iota(h_1))\neq w=s(\iota(h_2))$. Denote the cycle starting from $h_2$ by $C_{\alpha}$. In this case, $s=\beta\alpha'$, $\varphi_s=0$.
		
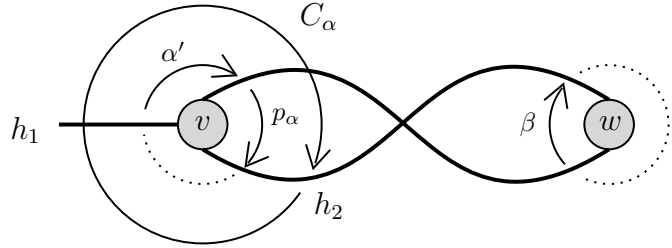
\begin{center}

			\tikzset{every picture/.style={line width=0.75pt}} 
			
			%
			
		\end{center}
		Then $\tphi_s=\sum_{i=0}^{m-1}\xi_i^{(s)}C_{\alpha}^i p_\alpha\alpha'$. In this case,
		$$\xi_i^{(s)'}-\xi_i^{(s)}=b_{i}^{(\beta)}\;(0\leq i\leq m-1).$$
		
		\item Consider (10) in Section 3.1. Let $\m(v)=m$, $\m(w)=n$, $\mathrm{val}(v)=\mathrm{val}(w)=1$, $e$ be the idempotent corresponding to this edge. Actually, $\alpha^m=\beta^n$ in $B_\Gamma$. In this case, $s=\beta\alpha$, $\varphi_s=0$.
		\begin{center}
			
			\tikzset{every picture/.style={line width=0.75pt}} 
			
			%
		\end{center}
		Then $\tphi_s=\lambda^{(s)}e+\sum_{i=1}^{m-1}\xi_i^{(s)}\alpha^i+\sum_{j=1}^{n-1}\zeta_j^{(s)}\beta^{j}+\kappa^{(s)}\beta^n$. In this case,  		$$\xi_i^{(s)'}-\xi_i^{(s)}=b_{i-1}^{(\beta)}\;(2\leq i\leq m-1),\;\zeta_j^{(s)'}-\zeta_j^{(s)}=b_{j-1}^{(\alpha)}\;(2\leq j\leq n-1),$$
		$$\xi_1^{(s)'}-\xi_1^{(s)}=a_{0}^{(\beta)},\;\zeta_1^{(s)'}-\zeta_1^{(s)}=a_{0}^{(\alpha)},\;\kappa^{(s)'}-\kappa^{(s)}=b_{n-1}^{(\alpha)}+b_{m-1}^{(\beta)}.$$
		
		\item  Consider (11) in Section 3.1. Let $\m(v)=m$, $\mathrm{val}(v)\neq 1=\mathrm{val}(w)$. In this case, $s=\beta\alpha$, $\varphi_s=0$.
		
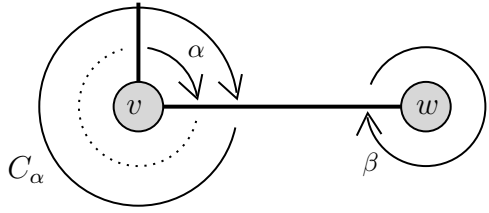
\begin{center}

			\tikzset{every picture/.style={line width=0.75pt}} 
			


		\end{center}
		Then $\tphi_s=\sum_{i=0}^{m-1}\xi_i^{(s)}C_\alpha^i\alpha$. In this case,  		$$\xi_i^{(s)'}-\xi_i^{(s)}=b_{i}^{(\beta)}\;(1\leq i\leq m-1),\;\xi_0^{(s)'}-\xi_0^{(s)}=a_{0}^{(\beta)}.$$
		
		\item Consider (12) in Section 3.1.  Let $\m(w)=n$, $\mathrm{val}(w)\neq 1=\mathrm{val}(v)$. In this case, $s=\beta\alpha$, $\varphi_s=0$.
		
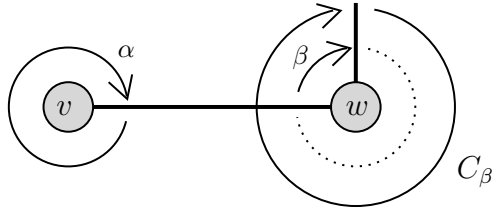
\begin{center}

			\tikzset{every picture/.style={line width=0.75pt}} 
			
			%

		\end{center}
		Then $\tphi_s=\sum_{j=0}^{n-1}\zeta_j^{(s)}C_\beta^{j}\beta$. In this case,  		$$\zeta_j^{(s)'}-\zeta_j^{(s)}=b_{j}^{(\alpha)}\;(1\leq j\leq n-1), \;\zeta_0^{(s)'}-\zeta_0^{(s)}=a_{0}^{(\alpha)}.$$
	\end{enumerate}
	
	\begin{remark}\label{rmk:coboundaries}
		\begin{itemize}
			\item By definition, each edge $\bar{h}$ corresponds to a commutative relation of the form $C_\alpha^{\m(C_\alpha)}=C_\beta^{\m(C_\beta)}$ in $B_\Gamma$. By Appendix \ref{appendix:2-cocycle}, each such commutative relation induces a $2$-cocycle sending $C_\alpha^{\m(C_\alpha)}$ to $ \kappa C_\beta^{\m(C_\beta)}$ with coefficient $\kappa \in \Bbbk$. For each cycle consisting of $n$ distinct vertices and $n$ distinct edges $\bar{h}_1,\cdots, \bar{h}_n$ in the ribbon graph $\Gamma$, denote by $\kappa_1,\cdots,\kappa_n$ the coefficients of the associated $2$-cocycles. Then Case 1 demonstrates that up to coboundary we may take $\kappa_1,\cdots,\kappa_{n-1}$ to be zero, while the term with coefficient $\kappa_n$ becomes a basis element of $\mathrm{HH}^2(B_\Gamma)$.
			
			\item Cases 2-8 show that each cocycle given by the terms with coefficients $\xi$ and $\zeta$ is a coboundary.
		\end{itemize}
	\end{remark}

    \vskip 5pt

\noindent {\bf Acknowledgements.}  \;This project is supported by National Key R\&D Program of China (No. 2024YFA1013803), the Fundamental Research
 Funds for the Central Universities (No. 020314380037), the National Natural
 Science Foundation of China (No. 13004005) and  the China Scholarship Council (No. 202506040127).


\end{document}